\documentclass[12pt]{article}

\usepackage[T1]{fontenc}
\usepackage[active]{srcltx}
\usepackage{lmodern}
\usepackage{amsmath,amstext,amssymb,graphicx,graphics,color}
\usepackage{mathtools,mathrsfs}
\usepackage{tikz}
\usetikzlibrary{decorations.markings,calc,fit,intersections,arrows,matrix,trees,positioning,shapes,tikzmark,backgrounds}
\usepackage{mleftright}
\usepackage{pgfplots}
\usepackage{theorem}
\usepackage{cite}               % to cite in increasing order
\usepackage{hyperref}           % faire un pdf avec des liens dynamiques
\usepackage{bbm}                % \mathbbm{1}
\usepackage{fancybox}           % pour \ovalbox
\usepackage[section]{placeins}  % \FloatBarrier
\usepackage[font=footnotesize]{caption}
\usepackage{subcaption}
\usepackage{enumitem}

\DeclarePairedDelimiter\floor{\lfloor}{\rfloor}

\usepackage{arydshln} % For picturing solid and dashed lines in matrix block.
\tikzset{
    invisible/.style={opacity=0},
    visible on/.style={alt={#1{}{invisible}}},
    alt/.code args={<#1>#2#3}{%
      \alt<#1>{\pgfkeysalso{#2}}{\pgfkeysalso{#3}}%
  }
}
\theorembodyfont{\normalfont\slshape} % on remet la config normal

\newtheorem{theorem}{Theorem}
\newtheorem{proposition}[theorem]{Proposition}
\newtheorem{corollary}[theorem]{Corollary}

\newtheorem{lemma}[theorem]{Lemma}
\theoremstyle{break} % les theorem vont à la ligne dorénavant ("plain" otherwise)

\newenvironment{proof}%
{{\par\noindent \bf Proof. \nobreak}}%
{\nobreak \removelastskip \nobreak \hfill $\Box$ \medbreak}
\newenvironment{remark}{\par \medskip \noindent {\bf Remark. }\nobreak}{\par \medskip}

\usepackage{fancyhdr}
 \fancypagestyle{plain}{
  \fancyhead{}
  
}

\def\paragraph#1{{\bf #1\ }}

\newcommand{\expo}{\mathrm{e}}

\newcommand{\Var}{\mathrm{Var}}

\newcommand{\dd}{\mathrm{d}}
\newcommand{\EE}{\mathbb{E}}
\newcommand{\R}{{\mathbb{R}}}

\newcommand{\cS}{\mathcal{S}}
\newcommand{\cA}{\mathcal{A}}
\newcommand{\cP}{\mathcal{P}}
\newcommand{\cB}{\mathcal{B}}
\newcommand{\cO}{\mathcal{O}}
\newcommand{\cM}{\mathcal{M}}
\newcommand{\bX}{\mathbf{X}}
\newcommand{\ind}{\mathbbm{1}}
\newcommand{\GBC}{\mathrm{GBC}}

\usepackage[utf8]{inputenc}
\def\law{\mathcal{L}}
\def\supp{\operatorname{supp}}

\title{The fractal geometry of opinion formation}

\author{Fei Cao \footnotemark[1] \and Roberto Cortez \footnotemark[2]}

\begin{document}
\maketitle

\footnotetext[1]{Amherst College - Department of Mathematics, Amherst, MA 01002, USA}
\footnotetext[2]{Universidad Andres Bello, Departamento de Matemáticas, Sazié 2212, sexto piso, Santiago, Chile.}

\begin{abstract}
In this manuscript, we introduce and study a variant of the agent-based opinion dynamics proposed in a recent work \cite{cao_fractal_2024}, within the framework of an interacting multi-agent system, where agents are assumed to interact with each other and update their opinions after each pairwise encounter. Specifically, our opinion model involves a large crowd of $N$ indistinguishable agents, each characterized by an opinion value ranging within the interval $[-1,1]$. At each update time, two agents are picked uniformly at random and the opinion of one agent will either shift by a proportion $\mu \in (0,1]$ towards $+1$, or by a proportion $\lambda \in (0,1]$ towards $-1$, with probabilities depending on the other agent's opinion. We rigorously derive the mean-field limit PDE that governs the large-population limit of the agent-based model and present several quantitative results demonstrating convergence to the unique equilibrium distribution. Remarkably, for a suitable choice of model parameters, the long-term equilibrium opinion profile displays a striking self-similar structure that generalizes the celebrated Bernoulli convolution, a topic extensively studied in the context of fractal geometry \cite{erdos_family_1939,varju_recent_2016}. These findings also enhance our understanding of the opinion fragmentation phenomenon and may provide valuable insights for the development of more sophisticated models in future research.
\end{abstract}


\tableofcontents

%%%%%%%%%%%%%%%%%%%%%%%%%%%%%%%%%%%%%%%%%%%%%%%%%%%%%%%%%%%%%%%%%%%%%%
%%%%%%%%%%%%%%%%%%%%%%%%%%%%%%%%%%%%%%%%%%%%%%%%%%%%%%%%%%%%%%%%%%%%%%
\section{Introduction}\label{sec:sec1}
\setcounter{equation}{0}
%%%%%%%%%%%%%%%%%%%%%%%%%%%%%%%%%%%%%%%%%%%%%%%%%%%%%%%%%%%%%%%%%%%%%%
%%%%%%%%%%%%%%%%%%%%%%%%%%%%%%%%%%%%%%%%%%%%%%%%%%%%%%%%%%%%%%%%%%%%%%
Recently, opinion dynamics are receiving increasing attention and are widely applied in areas such as political science, internet culture studies, and epidemic control. The mathematical study of how individuals form opinions and influence each other within a population dates back to at least the mid-1960s. Over the past few decades, the integration of physics-inspired methods into the social sciences has opened new avenues for modeling complex collective social and economic behaviors. This interdisciplinary approach has given rise to the fields of sociophysics and econophysics, both of which have drawn extensively on the tools of statistical physics  \cite{bennaim_2005,toscani_2006}. Sociophysics, introduced in \cite{galam_gefen_shapir_1982}, aims to unravel the dynamics of human social behavior through probabilistic and dynamical system frameworks. Since the early 2000s, the field has undergone rapid expansion, marked by the introduction and rigorous study of several landmark models \cite{deffuant_mixing_2000,hegselmann_opinion_2002,sznajd_opinion_2000}. While the literature on opinion dynamics is extensive and still growing \cite{cao_k_2021,cao_iterative_2024,castellano_statistical_2009,jabin_clustering_2014,sen_sociophysics_2014}, some of the most widely studied frameworks include the Deffuant model (also known as the bounded confidence model) \cite{deffuant_mixing_2000}, the Hegselmann-Krause model \cite{hegselmann_opinion_2002}, and the Sznajd model \cite{sznajd_opinion_2000}, along with their many generalizations.

The present work is primarily motivated by our recent research \cite{cao_fractal_2024}, where we introduced and analyzed a novel stochastic agent-based opinion model on the interval $[-1,1]$, using probabilistic and analytic tools. Notably, in that study, we identified a new mathematical description of the so-called opinion fragmentation phenomenon: under the (rigorous) large-population limit $N \to \infty$, the mean-field version of the model can give rise to a long-term equilibrium opinion profile whose support exhibits a fractal structure (under suitable choice of model parameters). The implications of these findings are significant, as these results suggest that in the long run, public opinions might become so polarized and fragmented that it becomes impossible for any agent to hold opinion values within certain subintervals of the opinion space $[-1,1]$.

%%%%%%%%%%%%%%%%%%%%%%%%%%%%%%%%%%%%%%%%%%%%%%%%%%%%%%%%%%%%%%%%%%%%%%
\subsection{Description of the model}
%%%%%%%%%%%%%%%%%%%%%%%%%%%%%%%%%%%%%%%%%%%%%%%%%%%%%%%%%%%%%%%%%%%%%%

We study an agent-based opinion dynamics model framed as an interacting multi-agent system, where individuals adjust their opinions through pairwise interactions. Specifically, consider a population of individuals/agents of size $N \in \mathbb{N}_+$. At any given time, each (indistinguishable) agent is uniquely characterized by her general opinion, or political standpoint, on a given issue, represented as a scalar ranging from $-1$ to $1$. We denote by $X^{i,N}_t \in [-1,1]$ the opinion of agent $i$ at time $t\geq 0$. A convenient analogue with terminologies from political sciences also enables us to interpret $-1$ and $1$ as representing extreme left-wing and right-wing positions, respectively. The dynamics of our agent-based model are described as follows:

\begin{itemize}
\item At each random time generated by a Poisson clock with rate $N/2$, select a pair of distinct agents $(i,j) \in \{1,\cdots,N\}^2 \setminus \{i=j\}$ uniformly at random and independently from the selection history. This guarantees the Markov property of the dynamics and also ensures that each agent interacts with all other agents at rate one. In each interaction, agent $j$ will state an opinion, which can be either $-1$ or $+1$, whereas agent $i$ acts as the ``listener''.
\item With probability $\left(1+X_{t^-}^{j,N}\right)/2$, agent $j$ states the opinion $+1$, and with the complementary probability $\left(1-X_{t^-}^{j,N}\right)/2$, agent $j$ states the opinion $-1$.
\item If agent $j$ expresses the opinion $+1$, agent $i$ updates her opinion by shifting a fixed proportion $\lambda \in (0,1]$ closer to $-1$.
\item If agent $j$ expresses the opinion $-1$, agent $i$ updates her opinion by shifting a fixed proportion $\mu \in (0,1]$ closer to $+1$.
\end{itemize}

Mathematically, if the pair of agents $(i,j)$ is chosen to interact at time $t$, then the opinion of agent $i$ will be updated according to
\begin{equation}
\label{eq:dynamics}
X_t^{i,N} = \begin{cases}
X_{t^-}^{i,N} - \lambda\cdot\left(1 + X_{t^-}^{i,N} \right) &~~ \textrm{with probability}~~ \frac 12 + \frac{X_{t^-}^{j,N}}{2}, \\
X_{t^-}^{i,N} + \mu\cdot\left(1 - X_{t^-}^{i,N} \right) &~~ \textrm{with probability}~~ \frac 12 - \frac{X_{t^-}^{j,N}}{2},
\end{cases}
\end{equation}
in which $\lambda \in (0,1]$ and $\mu \in (0,1]$ are user-specified model parameters which control the rate at which agents adjust their opinions toward $-1$ and $1$, respectively. By the obvious symmetry, we can and shall assume that $\lambda \leq \mu$ without any loss of generality.

We will assume throughout this article that the collection of initial conditions $X_0^{1,N},\ldots,X_0^{N,N}$ are independent and $\rho_0$-distributed, where $\rho_0$ is a given probability measure on $[-1,1]$. Denote ${\bf X}_t^N = (X_t^{1,N},\ldots,X_t^{N,N})$. We observe that the collection $(X_t^{1,N},\ldots,X_t^{N,N})$ is clearly exchangeable for any $t \geq 0$.

It is worth mentioning that the agent-based opinion model proposed and analyzed in the recent work \cite{cao_fractal_2024} is given by the update rules
\begin{equation}
\label{eq:dynamics_old}
X_t^{i,N} = \begin{cases}
X_{t^-}^{i,N} + \mu\cdot\left(1 - X_{t^-}^{i,N} \right) &~~ \textrm{with probability}~~ \frac 12 + \frac{X_{t^-}^{j,N}}{2}, \\
X_{t^-}^{i,N} - \lambda\cdot\left(1 + X_{t^-}^{i,N} \right) &~~ \textrm{with probability}~~ \frac 12 - \frac{X_{t^-}^{j,N}}{2},
\end{cases}
\end{equation}
which appear to be quite similar to the dynamics \eqref{eq:dynamics} investigated in the present article, except that the steps toward $-1$ and $+1$ (or, equivalently, the corresponding probabilities) have been exchanged. However, this similarity is only apparent. Firstly, the interactions can be interpreted in a completely opposite way: in the model \eqref{eq:dynamics_old}, agent $i$ moves \emph{toward} the opinion stated by $j$, whereas in \eqref{eq:dynamics} the agent moves \emph{away} from it. In other words, stated opinions either \emph{persuade} or \emph{repel} the listener, respectively. In the language of sociologists \cite{heinrich_conformity_2025}, the dynamics \eqref{eq:dynamics_old}, as well as most opinion models \cite{cao_k_2021,castellano_statistical_2009,deffuant_mixing_2000,hegselmann_opinion_2002,jabin_clustering_2014,sznajd_opinion_2000}, exhibit \emph{conformity}, whereas the model \eqref{eq:dynamics} studied in this manuscript exhibits \emph{anticonformity}. The latter type of behavior appears to be at least partially supported by recent studies in the social sciences that examine how people react to political disagreement and the influence of extreme or uncivil opinions \cite{barnidge_2018,goyanes_etal_2021,pandey_etal_2023,zhang_shoenberger_2024}. Secondly, from a mathematical point of view, the long-time behavior of these two dynamics differs significantly. For instance, as we shall see, in the large-population limit, the dynamics \eqref{eq:dynamics} admits a nontrivial equilibrium distribution for any values of $\lambda$ and $\mu$, whereas \eqref{eq:dynamics_old} admits such an equilibrium only when $\lambda = \mu$ (see \cite{cao_fractal_2024}). Thus, the long-time behavior of the proposed dynamics \eqref{eq:dynamics} is much more robust with respect to the model parameters $\lambda$ and $\mu$, giving rise to a richer family of equilibrium distributions with interesting properties, as we will discuss shortly.

%%%%%%%%%%%%%%%%%%%%%%%%%%%%%%%%%%%%%%%%%%%%%%%%%%%%%%%%%%%%%%%%%%%%%%
\subsection{Main results and overview of the paper}
%%%%%%%%%%%%%%%%%%%%%%%%%%%%%%%%%%%%%%%%%%%%%%%%%%%%%%%%%%%%%%%%%%%%%%

A central objective of this work is to examine the proposed opinion dynamics \eqref{eq:dynamics} through a kinetic perspective, which involves investigating the mean-field limit as $N \to \infty$ of the agent-based model \eqref{eq:dynamics}. This yields a Boltzmann-type PDE for the evolution of the distribution of opinions in an infinite population, whose asymptotic behavior as $t \to \infty$ can then be explored. We encapsulate the schematic illustration of the strategy used in this manuscript in Figure \ref{fig:scheme_sketch}.

%\begin{figure}[!htb]
\begin{figure}
\centering
\includegraphics[scale=0.8]{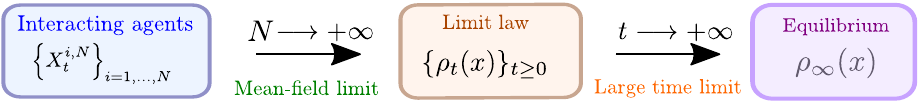}
\caption{Schematic illustration of the limiting procedure carried out for the study the multi-agent opinion dynamics \eqref{eq:dynamics}. We begin by establishing a propagation of chaos result in the large-population regime as $N \to \infty$, wherein individual interactions between agents become negligible, giving rise to a mean-field description governed by a Boltzmann-type PDE. Once this limit is formalized, we derive a series of quantitative estimates that substantiate the convergence of the mean-field PDE solution toward its unique equilibrium distribution.}
\label{fig:scheme_sketch}
\end{figure}

More specifically, the remainder of the paper is organized as follows: Section~\ref{sec:sec2} is devoted to the investigation of the mean-field limit of our opinion model \eqref{eq:dynamics} under the large-population limit $N \to \infty$, which allows us to link the stochastic multi-agent opinion dynamics \eqref{eq:dynamics} to a Boltzmann-type PDE \eqref{eq:PDE}. Heuristically speaking, our justification of the so-called \emph{propagation of chaos} phenomenon (see Theorem \ref{thm:PoC}) implies that as the population size $N$ tends to infinity, any subset of $k$ agents from the total population of size $N$ becomes more and more independent, and the evolution of each agent's opinion can be well-approximated to a (non-interacting) nonlinear limiting process $\{Z_t\}_{t\geq 0}$ \eqref{eq:SDE_linear}. Moreover, a Boltzmann-type PDE \eqref{eq:PDE} satisfied by the law of $Z_t$ (denoted by $\rho_t$) at each fixed time is established. As a consequence, we can also evaluate the first moment along the solution to the mean-field PDE \eqref{eq:PDE} explicitly, which underscores several fundamental distinctions relative to the dynamics \eqref{eq:dynamics_old} studied in \cite{cao_fractal_2024}.

Once the rigorous mean-field limit of the multi-agent system \eqref{eq:dynamics} is carried out, in Section~\ref{sec:sec3} we perform a large-time asymptotic analysis of the solution $\rho_t$ to the mean-field PDE \eqref{eq:PDE}. In particular, we prove several quantitative convergence guarantees regarding the large time convergence of solutions of \eqref{eq:PDE} to its (unique) equilibrium distribution (denoted by $\rho_\infty$). Our quantitative convergence analysis of the mean-field PDE \eqref{eq:PDE} involves various probabilistic and analytic tools, such as Wasserstein distances and Fourier-based metrics.

In Section~\ref{sec:sec4} we focus on the study of the equilibrium opinion distribution $\rho_\infty$ associated to the Boltzmann-type PDE \eqref{eq:PDE}. We show that, when the tendency of agents to move toward the extreme opinions $\pm 1$ is sufficiently strong---that is, when $\lambda + \mu >1$---the equilibrium distribution is supported on a Cantor-like fractal set, and we compute its Hausdorff dimension. As discussed in \cite{cao_fractal_2024} for the special case $\lambda = \mu > 1/2$, this can be seen as a manifestation of the aforementioned \emph{opinion fragmentation} phenomenon: the empirically observed tendency of the agents' opinions to cluster together, rather than spread continuously on $[-1,1]$. In the special case where $\lambda = \mu$, we recover the same stationary distribution as demonstrated in the recent work \cite{cao_fractal_2024}, which coincides with the \emph{Bernoulli convolution}, studied extensively in the fractal geometry literature \cite{erdos_family_1939,kershner_symmetric_1935,solomyak_random_1995,varju_recent_2016,varju_absolute_2019}. On the other hand, when the model parameters $\lambda,\mu$ are no longer equal, we obtain a family of non-trivial distributions $(\rho_\infty)_{\lambda,\mu}$, which generalizes the classic Bernoulli convolution and which has also been studied in the fractal geometry literature under the more general framework of iterative function systems \cite{hochman_AnnMath_memoir_2015,saglietti_etal_2018}. However, due to the dynamical nature of our model within the framework of interacting particle systems, the equilibrium object obtained in this work is still a genuinely interesting distribution that has not been studied systematically in the fractal geometry literature (to the best of our knowledge).

Finally, we conclude the present paper in Section~\ref{sec:sec5} where we summarize the main contributions of this paper and outline several compelling and technically challenging avenues for future research inspired by the opinion dynamics investigated in this work.

%%%%%%%%%%%%%%%%%%%%%%%%%%%%%%%%%%%%%%%%%%%%%%%%%%%%%%%%%%%%%%%%%%%%%%
%%%%%%%%%%%%%%%%%%%%%%%%%%%%%%%%%%%%%%%%%%%%%%%%%%%%%%%%%%%%%%%%%%%%%%
\section{Derivation of the mean-field opinion dynamics}\label{sec:sec2}
\setcounter{equation}{0}

%%%%%%%%%%%%%%%%%%%%%%%%%%%%%%%%%%%%%%%%%%%%%%%%%%%%%%%%%%%%%%%%%%%%%%
\subsection{Mean-field limit for general $\mu$ and $\lambda$}\label{subsec:2.1}

The \emph{mean-field limit} of the system $\bX^N$ captures the behavior of any fixed agent $X_t^{i,N}$, in the large-population limit as $N\to\infty$. It can refer to either a stochastic process $(Z_t)_{t\geq 0}$, or its collection of marginal distributions $(\rho_t:=\law(Z_t))_{t\geq 0}$. The mean-field process is typically described by a jump SDE driven by Poisson point measures, which we now derive informally.

It is straightforward to verify that the agent-based dynamics \eqref{eq:dynamics} is equivalent (in law) to the following system of SDEs: for each $i=1,\ldots,N$,
\begin{equation}
\label{eq:SDE_Xi}	
\begin{aligned}
	\dd X_t^{i,N}
	&= \int_0^1 \left[\mu (1-X_{t^-}^{i,N}) \mathbbm{1} \left\{ u < \tfrac{1-\cA_{t^-}^{i,N}}{2} \right\} \right. \\
	& \qquad \qquad \qquad  \left. {} - \lambda (1+X_{t^-}^{i,N}) \mathbbm{1} \left\{ u \geq \tfrac{1-\cA_{t^-}^{i,N}}{2} \right\}
	\right] \cP^i(\dd t, \dd u),
\end{aligned}
\end{equation}
where $\cP^i(\dd t, \dd u)$ is a Poisson point measure on $[0,\infty) \times [0,1]$ with intensity $\dd t \, \dd u$, the collection $(\cP^i)_{i=1}^N$ is independent, and $\cA_t^{i,N}$ denotes the average opinion of all agents other than $i$:
\[
	\cA_t^{i,N} \coloneqq \frac{1}{N-1} \sum_{\substack{j=1 \\ j\neq i}}^N X_t^{j,N}
\]
By examining \eqref{eq:SDE_Xi}, it is natural to conjecture that the mean-field process $(Z_t)_{t\geq 0}$, which describes the evolution of the opinion of a typical agent as the total number of agents goes to infinity (i.e., $N \to \infty$),  should be described by the following SDE:
\begin{equation}
	\label{eq:SDE_linear}	
	\begin{aligned}
		\dd Z_t
		&= \int_0^1 \left[\mu (1-Z_{t^-}) \mathbbm{1} \left\{ u < \tfrac{1-m_t}{2} \right\} \right. \\
		& \qquad \qquad \qquad  \left. {} - \lambda (1+Z_{t^-}) \mathbbm{1} \left\{ u \geq \tfrac{1-m_t}{2} \right\}
		\right] \cP(\dd t, \dd u),
	\end{aligned}
\end{equation}
where $\cP(\dd t, \dd u)$ is a Poisson point measure with intensity $\dd t \, \dd u$, and $m_t \coloneqq \EE[Z_t]$. Since $m_t$ depends on the law of $Z_t$, this SDE is nonlinear, in principle. However, $m_t$ can be computed explicitly, which turns \eqref{eq:SDE_linear} into a linear and time-inhomogeneous SDE. Therefore, it admits a unique strong solution $(Z_t)_{t\geq 0}$.

More specifically: by taking expectations in \eqref{eq:SDE_linear}, it is straightforward to see that $m_t$ must satisfy the following autonomous ODE:
\[\frac{\dd}{\dd t} m_t = \frac{\mu - \lambda}{2}\,(1 + m^2_t) - (\mu + \lambda)\,m_t,\]
which leads us to
\begin{equation}
\label{eq:mt_formula}
	m_t =  \begin{cases}
		m_0\,\expo^{-2\,\mu\,t},&~~\textrm{if~ $\lambda = \mu$},\\
		m_\infty + \frac{\gamma + \sqrt{\gamma^2 - 1} - m_\infty}{1+C\,\expo^{2\,\sqrt{\mu\,\lambda}\,t}}, &~~\textrm{if~ $\lambda \neq \mu$},
	\end{cases}
\end{equation}
where
\[
	m_\infty
	\coloneqq \frac{\sqrt{\mu}-\sqrt{\lambda}}{\sqrt{\mu}+\sqrt{\lambda}},
	\qquad
	\gamma
	\coloneqq \frac{\mu + \lambda}{\mu - \lambda},
	\qquad
	C
	\coloneqq \frac{\gamma + \sqrt{\gamma^2 - 1} - m_\infty}{m_0 - m_\infty} - 1.
\]
Notice that $m_t \to m_\infty$, where $m_\infty$ can be any number in $[0,1)$, depending on the specific values of the underlying model parameters $0<\lambda \leq \mu\leq 1$. This is a key difference from the model studied in \cite{cao_fractal_2024}, for which $m_\infty = 1$ whenever $\lambda < \mu$, and $m_\infty = m_0$ when $\lambda = \mu$.

The rigorous convergence of the system $\bX^N$ to its mean-field limit as $N\to\infty$, is called \emph{propagation of chaos} \cite{sznitman_topics_1991}. It has been studied extensively for a wide variety of systems arising from physical, social and economic sciences, especially in the context of kinetic models of Boltzmann type, see for instance \cite{cao_derivation_2021,cao_entropy_2021,cao_explicit_2021,cao_interacting_2022,cao_k_2021,cao_uncovering_2022,cao_uniform_2024,cortez_quantitative_2016,cortez_uniform_2016,cortez_fontbona_2018,degond_macroscopic_2004,graham_meleard_1997,holley_ergodic_1975,jabin_clustering_2014,liggett_interacting_1985}. For the model of the present article, we provide the following quantitative propagation of chaos result. The proof follows along the same lines as the one given in \cite{cao_fractal_2024} for a related model, so we omit it here. Denote by $W_p$ the Wasserstein metric of order $p\geq 1$ between probability measures on $\R^N$ with respect to the normalized distance $d({\bf x},{\bf y}) = \left(\frac{1}{N}\,\sum_{i=1}^N |x^i - y^i|^p\right)^{1/p}$.

\begin{theorem}[Propagation of chaos]
	\label{thm:PoC}
	There exists a universal constant $C>0$ such that for all $t>0$, it holds that
	\begin{equation}
		\label{eq:PoC}
		W_1 \left( \law( {\bf X}_t^N) \, , \, \rho_t^{\otimes N} \right)
		\leq \frac{C}{\sqrt{N}} \, \expo^{ (\mu + \lambda)t}.
	\end{equation}
\end{theorem}

\begin{remark}
We emphasize that Theorem \ref{thm:PoC} is only a finite in time propagation of chaos result since the right hand side of the bound \eqref{eq:PoC} deteriorates exponentially with increasing time $t$. We will significantly improve such finite in time propagation of chaos to a uniform in time estimate in the special case when $\lambda = \mu$ in Theorem \ref{thm:UPoC} later.
\end{remark}

From the SDE \eqref{eq:SDE_linear}, we can also obtain the associated Kolmogorov backward and forward equations governing the evolution of $\mathbb{E}[\varphi(Z_t)]$ (where $\varphi$ denotes a generic test function) and the law $\rho_t$ of $Z_t$, respectively. Indeed, for any bounded and continuous test function $\varphi$, we have
\begin{equation*}
%\label{eq:PDE_weak}
\begin{split}
&\frac{\dd}{\dd t} \int_{-1}^1 \varphi(x) \rho_t(\dd x) \\
&= \int_{-1}^1 \left[ \varphi(x + \mu (1-x))\, \frac{1-m_t}{2}
+ \varphi(x - \lambda (1+x)) \frac{1+m_t}{2}
- \varphi(x) \right] \rho_t(\dd x), \\
&\coloneqq \int_{-1}^1 Q_t[\varphi](x)\,\rho_t(\dd x)
\end{split}
\end{equation*}
where the operator $Q_t[\cdot]$, typically called the (infinitesimal) generator of the process $Z_t$, is defined via
\[
Q_t[\varphi](x) \coloneqq \varphi(x + \mu (1-x))\, \frac{1-m_t}{2} + \varphi(x - \lambda (1+x))\,\frac{1+m_t}{2} - \varphi(x)
\]
for all $(x,t) \in [-1,1] \times [0,\infty)$.

Assuming that the law of $Z_t$ admits a density which is still denoted by $\rho_t$ (upon slight abuse of notation), the Kolmogorov forward equation which describes the evolution of $\rho_t$ can be readily derived as well. Indeed, we observe that
\begin{equation*}
\int_{-1}^1 \varphi\left(x+\mu\cdot(1-x)\right)\,\rho_t(x)\,\dd x = \int_{2\mu-1}^1 \varphi(y)\,\rho_t\left(\frac{y-\mu}{1-\mu}\right)\,\frac{\dd y}{1-\mu}
\end{equation*}
together with
\begin{equation*}
\int_{-1}^1 \varphi\left(x-\lambda\cdot(1+x)\right)\,\rho_t(x)\,\dd x = \int_{-1}^{1-2\lambda} \varphi(y)\,\rho_t\left(\frac{y+\lambda}{1-\lambda}\right)\,\frac{\dd y}{1-\lambda},
\end{equation*}
thus the evolution of $\rho_t$ is governed by the following Boltzmann-type PDE (which should be interpreted in the weak sense):
\begin{equation}
\label{eq:PDE}
\partial_t \rho_t(x) = Q_+[\rho_t](x) - \rho_t(x)
\end{equation}
where
\begin{equation}
\label{eq:operator_Q+}
\begin{aligned}
Q_+[\rho_t](x) &= \frac{1-m_t}{2}\,\frac{\mathbbm{1}\{x> 2\mu -1\}}{1-\mu}\,\rho_t\left(\frac{x-\mu}{1-\mu}\right) \\
&\quad + \frac{1+m_t}{2}\,\frac{\mathbbm{1}\{x\leq 1-2\lambda\}}{1-\lambda}\,\rho_t\left(\frac{x+\lambda}{1-\lambda}\right).
\end{aligned}
\end{equation}

%%%%%%%%%%%%%%%%%%%%%%%%%%%%%%%%%%%%%%%%%%%%%%%%%%%%%%%%%%%%%%%%%%%%%%
\subsection{Uniform propagation of chaos when $\lambda = \mu$}\label{subsec:2.2}

We now aim to prove \emph{uniform-in-time} propagation of chaos of our multi-agent opinion dynamics in the special case when $\lambda = \mu$. We start by showing that in this case $\EE[X_t^{i,N}]$ coincides with $\EE[Z_t] = m_t = m_0\,\expo^{- 2\mu t}$ for all $t\geq 0$, regardless of the number of agents $N$.

\begin{proposition}
	Assume that $\lambda = \mu$. Then for any $1\leq i\leq N$,
	\[
	\EE[X_t^{i,N}] = m_0\,\expo^{ - 2\mu t}.
	\]
\end{proposition}

\begin{proof}
Denote $h(t) \coloneqq \EE[X_t^{i,N}]$. From \eqref{eq:SDE_Xi} we see that
\[
	\frac{\dd}{\dd t} h(t)
	= \EE\left[
	   \mu (1-X_{t}^{i,N}) \tfrac{1-\cA_{t}^{i,N}}{2}
	   - \mu (1+X_{t}^{i,N}) \tfrac{1+\cA_{t}^{i,N}}{2}
	   \right]
	= - \mu h(t) - \mu \EE[\cA_t^{i,N}].
\]
Since $\EE[\cA_t^{i,N}] = h(t)$ thanks to exchangeability, we arrive at $\frac{\dd}{\dd t} h(t) = -2\mu h(t)$ and the announced result follows.
\end{proof}

Our propagation of chaos analysis involves estimating $\EE[|\cA_t^{i,N} - m_t|]$, which is a law-of-large-numbers type estimate since $\EE[\cA_t^{i,N}] = m_t$. To this end, we introduce
\[
\cS_t^N \coloneqq \sum_{i=1}^N X_t^{i,N}.
\]
which represents the sum of opinion status of all agents at time $t$. The key computation is carried out in the following lemma:

\begin{lemma}
	\label{lem:ES2}
	Assume that $\lambda = \mu$. Then, there exists a constant $C = C(\mu)$ depending only on $\mu$ such that
	\[
		\EE[(\cS_t^N)^2]
		\leq \EE[(\cS_0^N)^2]\, \expo^{-4 \mu t} + CN.
	\]
\end{lemma}

\begin{proof}
For notational simplicity, we drop the $N$ in the superscript and simply write $\cS_t$ and $X_t^i$ instead of $\cS_t^N$ and $X_t^{i,N}$, respectively. Denote $h(t) \coloneqq \EE[\cS_t^2]$ to be the second (raw) moment of $\cS_t$. Setting aside the terms involving $X_t^i$ in the expansion of $\cS_t^2$, we have
\[\cS_t^2 = (X_t^i)^2 + 2\,X_t^i\, \sum_{\substack{j=1 \\ j\neq i}}^N X_t^j + \sum_{\substack{j,k=1 \\ j,k\neq i}}^N X_t^j\,X_t^k.\]
Notice that when the opinion of agent $i$ jumps, the last summation remains unchanged, whence it will cancel with the pre-jump summation. Consequently, again from \eqref{eq:SDE_Xi}, we obtain
\begin{align}
	\dd h(t)
	&= \EE \left[ \sum_{i=1}^N \int_0^1 \left( \cS_t^2 - \cS_{t^-}^2 \right) \cP^i(\dd t, \dd u) \right]
	\notag \\
	&= \EE \left[ \sum_{i=1}^N \int_0^1 \left(
	(X_t^i)^2 - (X_{t^-}^i)^2 + 2 (X_t^i - X_{t^-}^i) \sum_{j\neq i} X_{t^-}^j
	\right) \cP^i(\dd t, \dd u) \right].
	\label{eq:dES2}
\end{align}
Clearly we have
\begin{align*}
	X_t^i - X_{t^-}^i
	&= \mu(1-X_{t^-}^i) \ind \left\{ u < \tfrac{1-\cA_{t^-}^i}{2} \right\} - \mu(1+X_{t^-}^i) \ind \left\{ u \geq \tfrac{1-\cA_{t^-}^i}{2} \right\}
	\notag \\
	&= - \mu X_{t^-}^i + \mu \left[ \ind \left\{ u < \tfrac{1-\cA_{t^-}^i}{2} \right\}  -  \ind \left\{ u \geq \tfrac{1-\cA_{t^-}^i}{2} \right\} \right].
\end{align*}
Notice that the integral of the difference of indicators with respect to $\dd u$ gives $- \cA_{t^-}^i$. Since the intensity of $\cP^i(\dd t,\dd u)$ is $\dd t\, \dd u$, from \eqref{eq:dES2}, using that $(X_t^i)^2 - (X_{t^-}^i)^2 \leq 1$, we deduce that
\begin{align*}
	\frac{\dd}{\dd t} h(t)
	&\leq \EE \left[
	\sum_{i=1}^N \left(
	1 - 2\, \mu\,(X_{t}^i + \cA_{t}^i)\, \sum_{j\neq i} X_t^j
	\right)
	\right] \\
	&= \EE \left[
	\sum_{i=1}^N \left(
	1 - 2\,\mu\,(X_{t}^i + \cA_{t}^i)\, (\cS_t - X_t^i)
	\right)
	\right] \\
	&\leq \EE [ 5N - 4\,\mu\, \cS_t^2],
\end{align*}
where we have used that $\sum_i X_t^i = \cS_t = \sum_i \cA_t^i$ and $(X_t^i+\cA_t^i)X_t^i \leq 2$. Thus, we end up with the differential inequality $\frac{\dd}{\dd t} h(t) \leq 5N - 4\,\mu\, h(t)$, and the advertised bound follows readily from Grönwall's inequality.
\end{proof}

\begin{corollary}
	\label{cor:EAm}
Assume that $\lambda = \mu$. Then, there exists a constant $C = C(\mu)$ depending only on $\mu$ such that for all $1\leq i\leq N$ and $t\geq 0$,
	\[
	\EE[|\cA_t^{i,N} - m_t|]^2
	\leq
	\EE[(\cA_t^{i,N} - m_t)^2]
	\leq \frac{C}{N}.
	\]
\end{corollary}

\newcommand{\cV}{\mathcal{U}}
\begin{proof}
It suffices to prove that $\EE[(\cV_t^N-m_t)^2] \leq C/N$ for $\cV_t^N \coloneqq \frac{1}{N} \cS_t^N$. Since $\EE[\cV_t^N] = m_t = m_0\, \expo^{-2\mu t}$, invoking Lemma \ref{lem:ES2} after dividing by $N^2$, we obtain:
\begin{align*}
	\EE[(\cV_t^N - m_t)^2]
	&= \EE[(\cV_t^N)^2] - m_t^2 \\
	&\leq \EE[(\cV_0^N)^2]\, \expo^{-4\mu t} + \frac{C}{N} - m_0^2\, \expo^{-4\mu t} \\
	&\leq \EE[(\cV_0^N)^2 - m_0^2] + \frac{C}{N} \\
	&= \operatorname{Var}(\cV_0^N) + \frac{C}{N}.
\end{align*}
As $(X_0^{i,N})_{i=1}^N$ are i.i.d.\ with law $\rho_0$, we have $\operatorname{Var}(\cV_0^N) = \operatorname{Var}(\rho_0)/N$, which concludes the proof.
\end{proof}

We are now ready to prove the following uniform in time propagation of chaos result for our agent-based opinion dynamics when $\lambda = \mu$, which refines the finite in time propagation of chaos guarantee reported in Theorem \ref{thm:PoC}.

\begin{theorem}[Uniform propagation of chaos when $\lambda = \mu$]
	\label{thm:UPoC}
	Assume that $\lambda = \mu$. Then, for all $p \geq 1$, there exists a constant $C = C(\mu,p)$ depending only $\mu$ and $p$ such that for all $1\leq i\leq N$ and $t\geq 0$,
	\[
	W^p_p(\law(\bX_t^N), \rho_t^{\otimes N})
	\leq \frac{C}{\sqrt{N}}.
	\]
\end{theorem}

\begin{proof}
For the sake of notational simplicity, we drop the $N$ in the superscript again. We resort to a standard coupling argument inspired from several recent works  \cite{cao_derivation_2021,cortez_quantitative_2016,cortez_uniform_2016}, which couples the system $\bX = (X^1,\ldots,X^N)$ with a system $\mathbf{Z} = (Z^1,\ldots,Z^N)$ of independent mean-field processes. To be more precise, each $X^i$ is governed by the SDE \eqref{eq:SDE_Xi}, whereas $Z^i$ solves the mean-field SDE \eqref{eq:SDE_linear} with $\cP^i$ in place of $\cP$. That is:
	\[
		\dd X_t^i
		= \int_0^1 \left[\mu (1-X_{t^-}^i) \mathbbm{1} \left\{ u < \tfrac{1-\cA_{t^-}^i}{2} \right\}
		- \mu (1+X_{t^-}^i) \mathbbm{1} \left\{ u \geq \tfrac{1-\cA_{t^-}^i}{2} \right\}
		\right] \cP^i(\dd t, \dd u),
	\]
	and
	\[
	\dd Z_t^i
	= \int_0^1 \left[\mu (1-Z_{t^-}^i) \mathbbm{1} \left\{ u < \tfrac{1-m_t}{2} \right\}
	- \mu (1+Z_{t^-}^i) \mathbbm{1} \left\{ u \geq \tfrac{1-m_t}{2} \right\}
	\right] \cP^i(\dd t, \dd u).
	\]
	Recall that $(\cP^i)_{i=1}^N$ are independent Poisson point measures on $[0,\infty) \times [0,1]$ with intensity $\dd t \, \dd u$, and that $X_0^1,\ldots,X_0^N$ are i.i.d.\ and $\rho_0$-distributed. We set $\mathbf{Z}_0 = \bX_0$ to ensure that $Z^1,\ldots,Z^N$ are independent mean-field processes.
	
	For arbitrary but fixed $i\in\{1,\ldots,N\}$, let $h(t) \coloneqq \EE[ |X_t^i - Z_t^i|^p]$. Since $W_p$ is a coupling distance, we have
	\[
		W^p_p(\law(\bX_t^N), \rho_t^{\otimes N})
		\leq \EE\left[\frac{1}{N} \sum_{j=1} |X_t^j - Z_t^j|^p \right]
		= h(t),
	\]
	due to exchangeability. Thus, it suffices to bound $h(t)$ from above.
	
	Let $r_t^i \coloneqq \min(m_t, \cA_t^i)$ and $R_t^i \coloneqq \max(m_t, \cA_t^i)$, thus both $X^i$ and $Z^i$ jump towards $+1$ when $u<(1-R_{t^-}^i)/2$, and both jump towards $-1$ when $u\geq (1-r_{t^-}^i)/2$. Otherwise, one of them jumps towards $+1$ while the other jumps towards $-1$, and we can simply upper bound the change of the value of $|X_t^i - Z_t^i|$ after a single jump by $2$. From the SDEs that define $X^i$ and $Z^i$, we thus deduce that
    \begin{equation}
    \label{eq:chain_of_estimates}
	\begin{aligned}
		\frac{\dd}{\dd t}h(t) + h(t)
		&\leq \EE \left[
		\left|X_t^i + \mu(1-X_t^i) - Z_t^i - \mu(1-Z_t^i) \right|^p \tfrac{1-R_t^i}{2} \right. \\
		& \left. \qquad {}
		+ \left|X_t^i - \mu(1+X_t^i) - Z_t^i + \mu(1+Z_t^i) \right|^p \tfrac{1+r_t^i}{2}
		+ 2^p \tfrac{R_t^i-r_t^i}{2}
		\right] \\
		&= \EE \left[
		(1-\mu)^p |X_t^i - Z_t^i|^p \left( 1 - \tfrac{R_t^i-r_t^i}{2} \right) + 2^{p-1} (R_t^i - r_t^i)
		\right] \\
		&\leq (1-\mu)^p h(t) + 2^{p-1} \EE[ |\cA_t^i - m_t|],
	\end{aligned}
    \end{equation}
where we used that $R_t^i - r_t^i = |\cA_t^i - m_t|$. Thanks to Corollary \ref{cor:EAm}, we know that $\EE[ |\cA_t^i - m_t|] \leq C/\sqrt{N}$, which leads us to
	\[
		\frac{\dd}{\dd t}h(t)
		\leq - \alpha h(t) + \frac{C}{\sqrt{N}}.
	\]
	for $\alpha = 1-(1-\mu)^p > 0$. As $\mathbf{Z}_0 = \bX_0$, we have $h(0) = 0$ and the conclusion follows readily from Grönwall's inequality.
\end{proof}

\begin{remark}
The uniform propagation of chaos guarantee reported in Theorem \ref{thm:UPoC} is a significant refinement of the previous finite in time result stated in Theorem \ref{thm:PoC}, and it justifies the use of a simplified mean-field PDE dynamics as a good approximation of the underlying stochastic agent-based dynamics even when time is very large.
\end{remark}

%%%%%%%%%%%%%%%%%%%%%%%%%%%%%%%%%%%%%%%%%%%%%%%%%%%%%%%%%%%%%%%%%%%%%%
\subsection{Long-time behavior of the finite system when $\lambda = \mu$}\label{subsec:2.3}

To conclude this section, we provide some estimates for the long-time behavior of the finite system. Ideally, one would want to prove that $\bX_t^N$ converges to a unique stationary distribution as $t\to\infty$ for each fixed $N$. Unfortunately, we were not quite able to prove it. However, the following result can be viewed as a first step towards that direction:

\newcommand{\tX}{\tilde{X}}
\newcommand{\tbX}{\tilde{\mathbf{X}}}
\newcommand{\tcA}{\tilde{\mathcal{A}}}
\begin{proposition}
Assume that $\lambda = \mu$. Denote $\tbX^N = (\tX^{1,N}, \ldots, \tX^{N,N})$ the solution of the system of SDEs \eqref{eq:SDE_Xi} with the same Poisson point measures $(\cP^i)_{i=1}^N$, starting with initial condition $\tbX_0^N \sim \tilde{\rho}_0^{\otimes N}$ for some distribution $\tilde{\rho}_0 \in \mathcal{P}([-1,1])$, possibly distinct from $\rho_0$. Then, for any $1\leq i\leq N$ it holds that
	\begin{itemize}
		\item[(i)] $\EE[|X_t^{i,N} - \tX_t^{i,N}|]$ is non-increasing with respect to $t$.
		
		\item[(ii)] There exists some constant $C=C(\mu)$ depending only on $\mu$ such that
		\[
			\lim_{t\to\infty} \EE[|X_t^{i,N} - \tX_t^{i,N}|] \leq \frac{C}{\sqrt{N}}.
		\]
	\end{itemize}
\end{proposition}

\begin{proof}
We drop the $N$ in the superscripts as usual. Denote $\tcA_t^i \coloneqq \frac{1}{N-1} \sum_{j\neq i} \tX_t^j$. Let $r_t^i \coloneqq  \min(\cA_t^i, \tcA_t^i)$ and $R_t^i \coloneqq  \max(\cA_t^i, \tcA_t^i)$. Examining \eqref{eq:SDE_Xi}, we see that when $u<(1-R_{t^-}^i)/2$, both $X^i$ and $\tX^i$ jump towards $+1$, and when $u\geq (1-r_{t^-}^i)/2$, both jump towards $-1$. Otherwise, $X^i$ will jump upwards and $\tX^i$ downwards, or vice versa. In the first case, we have:
\begin{align*}
	|X_t^i - \tX_t^i|
	&= |X^i_{t^-} + \mu (1-X_{t^-}^i) - \tX_{t^-}^i + \mu (1+\tX_{t^-}^i)| \\
	&\leq 2\mu + (1-\mu)|X_{t^-}^i - \tX_{t^-}^i|.
\end{align*}
In the second case, the same inequality holds. Thus, with a computation similar as the one presented in the proof of Theorem \ref{thm:UPoC}, we have for $h(t) \coloneqq \EE[|X_t^i - \tX_t^i|]$ that
\begin{align}
	\frac{\dd}{\dd t}h(t) + h(t)
	&\leq \EE \left[
	(1-\mu) |X_t^i - \tX_t^i| \left( 1 - \tfrac{R_t^i-r_t^i}{2} \right) + (2\mu + (1-\mu)|X_t^i - \tX_t^i|) \tfrac{R_t^i-r_t^i}{2}
	\right]
	\notag \\
	&= (1-\mu)\,h(t) + \mu\,\EE[ |\cA_t^i - \tcA_t^i|].
	\label{eq:dEAtildeA}
\end{align}
Clearly $\EE[ |\cA_t^i - \tcA_t^i|] \leq \frac{1}{N-1} \sum_{j\neq i} \EE[|X_t^j - \tX_t^j|] = h(t)$, which then yields $\frac{\dd}{\dd t}h(t) \leq 0$, proving that $h(t)$ is non-increasing.

Now we prove the second assertion: denoting $\tilde{m}_t = \EE[\tX_t^i] = \EE[\tcA_t^i] = \tilde{m}_0\, \expo^{-2\mu t}$, we have
\begin{align*}
	\EE[ |\cA_t^i - \tcA_t^i|]
	&\leq \EE[ |\cA_t^i - m_t|] + \EE[ |\tcA_t^i - \tilde{m}_t|] + |m_t - \tilde{m}_t| \\
	&\leq \frac{C}{\sqrt{N}} + |m_0-\tilde{m}_0|\,\expo^{-2\mu t},
\end{align*}
where we have used the content of Corollary \ref{cor:EAm}. From \eqref{eq:dEAtildeA}, we thus obtain
\[
	\frac{\dd}{\dd t}h(t) \leq -\mu h(t) + \frac{C}{\sqrt{N}} + |m_0-\tilde{m}_0|\expo^{-2\mu t}.
\]
The second assertion follows after applying Grönwall's inequality and taking the large time limit as $t\to\infty$.
\end{proof}

%%%%%%%%%%%%%%%%%%%%%%%%%%%%%%%%%%%%%%%%%%%%%%%%%%%%%%%%%%%%%%%%%%%%%%
%%%%%%%%%%%%%%%%%%%%%%%%%%%%%%%%%%%%%%%%%%%%%%%%%%%%%%%%%%%%%%%%%%%%%%
\section{Convergence to equilibrium for the mean-field PDE}
\label{sec:sec3}
\setcounter{equation}{0}
%%%%%%%%%%%%%%%%%%%%%%%%%%%%%%%%%%%%%%%%%%%%%%%%%%%%%%%%%%%%%%%%%%%%%%
%%%%%%%%%%%%%%%%%%%%%%%%%%%%%%%%%%%%%%%%%%%%%%%%%%%%%%%%%%%%%%%%%%%%%%

We now turn to the asymptotic analysis of the solution $\rho_t$ to \eqref{eq:PDE} as $t \to \infty$. For notational simplicity, we write $\mathcal{L}(X) = \rho$ to mean that the law of a real-valued random variable $X$ is $\rho$. We start with the following elementary yet important observation, which unveils a probabilistic interpretation of the collision gain operator $Q_+$ \eqref{eq:operator_Q+}.

\begin{lemma}\label{lem:Q_+}
Assume that $\mathcal{L}(Z) = \rho \in \mathcal{P}([-1,1])$ and $\mathcal{B} \sim \textrm{Bernoulli}\left(\frac{1-m}{2}\right)$ is a Bernoulli coin independent of $Z$, where $m$ is the mean of the law $\rho$. Then
\begin{equation*}
	%\label{eq:law_of_Q+}
	Q_+[\rho] = \mathcal{L}\left(\mu\,\mathcal{B} - \lambda\,(1-\mathcal{B}) + Z\,\left[(1-\mu)\,\mathcal{B} + (1-\lambda)\,(1-\mathcal{B})\right]\right).
\end{equation*}
In particular, when $\lambda = \mu$, we have
\begin{equation}\label{eq:law_of_Q+_equal_mu}
Q_+[\rho] = \mathcal{L}\left(\mu\,(2\,\mathcal{B} - 1) + (1-\mu)\,Z\right).
\end{equation}
\end{lemma}

The proof of Lemma \ref{lem:Q_+} consists of straightforward computations and hence will be omitted. In order to establish convergence to equilibrium associated to the solution of the Boltzmann-type equation \eqref{eq:PDE}, a crucial ingredient relies on the contractivity of the collision gain operator $Q_+$ in a suitable metric. However, since the mean-field dynamics \eqref{eq:PDE} does not preserve the mean value of the solution, which in the language of statistical physics literature implies that the dynamics is not strictly conservative, we only managed to establish some weak contractivity properities of the operator $Q_+$. Nevertheless, these weak contractivity results are already sufficient to deduce quantitative convergence guarantees for the solution of \eqref{eq:PDE} to its equilibrium distribution. We denote by $W_p(\cdot,\cdot)$ the $p$-Wasserstein distance between probability measures on $[-1,1]$.

\begin{proposition}[Weak contractivity of $Q_+$ in $W_1$ when $\lambda = \mu$]\label{prop:contra_W1_equal_mu}
Assume that $\lambda = \mu \in (0,1)$. Suppose that $Z^{(i)} \sim \rho^{(i)} \in \mathcal{P}([-1,1])$ with mean value $m^{(i)}$ for $i=1,2$, then
\begin{equation}\label{eq:W1_contra_equal_mu}
W_1\left(Q_+[\rho^{(1)}], Q_+[\rho^{(2)}]\right) \leq \mu\,|m^{(1)} - m^{(2)}| + (1-\mu)\,W_1(\rho^{(1)},\rho^{(2)}).
\end{equation}
Consequently, let $(\rho^{(1)}_t)_{t \geq 0}$ and $(\rho^{(2)}_t)_{t \geq 0}$ be the solutions to \eqref{eq:PDE} corresponding to initial datum $\rho^{(1)}_0$ and $\rho^{(2)}_0$ with mean values $m^{(1)}_0$ and $m^{(2)}_0$, respectively. Then for $t\geq 0$ it holds that
\begin{equation}\label{eq:W1_contra_equal_mu_PDE}
W_1\left(Q_+[\rho^{(1)}_t], Q_+[\rho^{(2)}_t]\right) \leq \mu\,|m^{(1)}_0 - m^{(2)}_0|\,\expo^{-2\,\mu\,t} + (1-\mu)\,W_1(\rho^{(1)}_t,\rho^{(2)}_t).
\end{equation}
\end{proposition}

\begin{proof}
It suffices to establish the estimate \eqref{eq:W1_contra_equal_mu} as the inequality \eqref{eq:W1_contra_equal_mu_PDE} follows readily from \eqref{eq:W1_contra_equal_mu} and the explicit formula \eqref{eq:mt_formula}. Thanks to the probabilistic interpretation \eqref{eq:law_of_Q+_equal_mu} of the collision gain operator $Q_+$, if we let $\mathcal{B}^{(1)} \sim \textrm{Bernoulli}\left(\frac{1-m^{(1)}}{2}\right)$ and $\mathcal{B}^{(2)} \sim \textrm{Bernoulli}\left(\frac{1-m^{(2)}}{2}\right)$ to be two Bernoulli coins such that $\mathcal{B}^{(i)}$ is independent of $Z^{(i)}$ for $i=1,2$, then
\[Q_+[\rho^{(i)}] = \mathcal{L}\left(\mu\,(2\,\mathcal{B}^{(i)} - 1) + (1-\mu)\,Z^{(i)}\right).\] As we can couple the Bernoulli coins $\mathcal{B}^{(1)}$ and $\mathcal{B}^{(2)}$ in such a way to ensure that \[\mathbb{E}\left|\mathcal{B}^{(1)} - \mathcal{B}^{(2)}\right| = \left|\frac{1-m^{(1)}}{2} - \frac{1-m^{(2)}}{2} \right| = \frac{|m^{(1)} - m^{(2)}|}{2}\] and that $\mathcal{B}^{(1)} - \mathcal{B}^{(2)}$ is independent of $Z^{(1)} - Z^{(2)}$, we deduce that
\begin{align*}
W_1\left(Q_+[\rho^{(1)}_t], Q_+[\rho^{(2)}_t]\right) &\leq \mathbb{E}\left|2\,\mu\,(\mathcal{B}^{(1)} - \mathcal{B}^{(2)}) + (1-\mu)\,(Z^{(1)} - Z^{(2)}) \right| \\
&= 2\,\mu\,\mathbb{E}\left|\mathcal{B}^{(1)} - \mathcal{B}^{(2)}\right| + (1-\mu)\,\mathbb{E}\left|Z^{(1)} - Z^{(2)}\right| \\
&= \mu\,|m^{(1)} - m^{(2)}| + (1-\mu)\,\mathbb{E}\left|Z^{(1)} - Z^{(2)}\right|,
\end{align*}
whence the advertised estimate \eqref{eq:W1_contra_equal_mu} follows by choosing the optimal coupling between $Z^{(1)}$ and $Z^{(2)}$ with respect to $W_1$.
\end{proof}

A similar argument also allows us to arrive at a variant of \eqref{eq:W1_contra_equal_mu} in which the $W_1$ distance is replaced by the squared Wasserstein distance of order $2$.

\begin{proposition}[Weak contractivity of $Q_+$ in $W_2$ when $\lambda = \mu$]\label{prop:contra_W2_equal_mu}
Under the settings of Proposition \ref{prop:contra_W1_equal_mu},
\begin{equation}\label{eq:W2_contra_equal_mu}
W^2_2\left(Q_+[\rho^{(1)}], Q_+[\rho^{(2)}]\right) \leq 2\,\mu^2\,|m^{(1)} - m^{(2)}| + (1-\mu)^2\,W^2_2(\rho^{(1)},\rho^{(2)}).
\end{equation}
As a consequence, let $(\rho^{(1)}_t)_{t \geq 0}$ and $(\rho^{(2)}_t)_{t \geq 0}$ be the solutions to \eqref{eq:PDE} corresponding to initial datum $\rho^{(1)}_0$ and $\rho^{(2)}_0$ with mean values $m^{(1)}_0$ and $m^{(2)}_0$, respectively. Then for $t\geq 0$ it holds that
\begin{equation}\label{eq:W2_contra_equal_mu_PDE}
W^2_2\left(Q_+[\rho^{(1)}_t], Q_+[\rho^{(2)}_t]\right) \leq 2\,\mu^2\,|m^{(1)}_0 - m^{(2)}_0|\,\expo^{-2\,\mu\,t} + (1-\mu)^2\,W^2_2(\rho^{(1)}_t,\rho^{(2)}_t).
\end{equation}
\end{proposition}

\begin{proof}
The proof is similar in spirit to the proof of \eqref{eq:W1_contra_equal_mu}, although some extra observations are required. Adopting the notations introduced in the proof of \eqref{eq:W1_contra_equal_mu}, we have
\begin{align*}
W^2_2\left(Q_+[\rho^{(1)}_t], Q_+[\rho^{(2)}_t]\right) &\leq \mathbb{E}\left[\left|2\,\mu\,(\mathcal{B}^{(1)} - \mathcal{B}^{(2)}) + (1-\mu)\,(Z^{(1)} - Z^{(2)}) \right|^2\right] \\
&= (2\,\mu)^2\,\mathbb{E}\left[\left|\mathcal{B}^{(1)} - \mathcal{B}^{(2)}\right|^2\right] + (1-\mu)^2\,\mathbb{E}\left[\left|Z^{(1)} - Z^{(2)}\right|^2\right] \\
&\quad + 4\,\mu\,(1-\mu)\,\mathbb{E}\left[\mathcal{B}^{(1)} - \mathcal{B}^{(2)}\right]\,\mathbb{E}\left[Z^{(1)} - Z^{(2)}\right]  \\
&\leq 2\,\mu^2\,|m^{(1)} - m^{(2)}| + (1-\mu)^2\,\mathbb{E}\left[\left|Z^{(1)} - Z^{(2)}\right|^2\right],
\end{align*}
where the inequality follows from the observation that
\begin{align*}
\mathbb{E}\left[\mathcal{B}^{(1)} - \mathcal{B}^{(2)}\right]\,\mathbb{E}\left[Z^{(1)} - Z^{(2)}\right] &= \left(\frac{1-m^{(1)}}{2} - \frac{1-m^{(2)}}{2}\right)\,\left(m^{(1)} - m^{(2)}\right) \\
&= -\frac 12\,\left|m^{(1)} - m^{(2)}\right|^2 \leq 0.
\end{align*}
Thus the announced estimate \eqref{eq:W2_contra_equal_mu} follows by coupling $Z^{(1)}$ and $Z^{(2)}$ in the optimal way with respect to $W_2$.
\end{proof}

We are now ready to establish the convergence of the solution of the mean-field PDE \eqref{eq:PDE} under the $W_2$ framework.

\begin{theorem}\label{thm:2}
Assume that $\lambda = \mu \in (0,1)$. Let $(\rho^{(1)}_t)_{t \geq 0}$ and $(\rho^{(2)}_t)_{t \geq 0}$ be the solutions to \eqref{eq:PDE} corresponding to initial datum $\rho^{(1)}_0$ and $\rho^{(2)}_0$ with mean values $m^{(1)}_0$ and $m^{(2)}_0$, respectively. Then for all $t\geq 0$ we have
\begin{equation}\label{eq:pairwise_conv_W2_equal_mu}
W^2_2\left(\rho^{(1)}_t, \rho^{(2)}_t\right) \leq \left(W^2_2\left(\rho^{(1)}_0, \rho^{(2)}_0\right) + 2\,\left|m^{(1)}_0 - m^{(2)}_0\right|\right)\,\expo^{-(2\mu - \mu^2)\,t}.
\end{equation}
In particular, for all $t\geq 0$ we have
\begin{equation}\label{eq:conv_W2_equal_mu}
W^2_2\left(\rho_t, \rho_\infty\right) \leq \left(W^2_2\left(\rho_0, \rho_\infty\right) + 2\,|m_0|\right)\,\expo^{-(2\mu - \mu^2)\,t}.
\end{equation}
\end{theorem}

\begin{proof}
The proof follows from a straightforward adaptation of a well-established procedure encountered in the study of the asymptotic behavior of dissipative kinetic equations \cite{bobylev_generalization_1992}, for which we refer the interested readers to the seminal work \cite{carrillo_contractive_2007} for more details. We recall that our mean-field PDE under investigation reads as
\begin{equation}\label{eq:main_PDE}
\frac{\partial \rho_t}{\partial t} = Q_+[\rho_t] - \rho_t.
\end{equation}
Consider the explicit Euler approximation to equation \eqref{eq:main_PDE}, given by
\[\rho_{t+\Delta t}(x) = \Delta t\,Q_+[\rho_t](x) + (1- \Delta t)\,\rho_t(x)\]
where $\Delta t \ll 1$. Due to the (joint) convexity of the squared Wasserstein distance of order 2 \cite{carrillo_contractive_2007}, the previous identity implies that
\begin{equation}\label{eq:E1}
W^2_2\left(\rho^{(1)}_{t+\Delta t}, \rho^{(2)}_{t+\Delta t}\right) \leq \Delta t\,W^2_2\left(Q_+[\rho^{(1)}_t], Q_+[\rho^{(2)}_t]\right) + (1-\Delta t)\,W^2_2\left(\rho^{(1)}_t, \rho^{(2)}_t\right).
\end{equation}
Inserting the estimate \eqref{eq:W2_contra_equal_mu_PDE} into \eqref{eq:E1} yields
\begin{align*}
&W^2_2\left(\rho^{(1)}_{t+\Delta t}, \rho^{(2)}_{t+\Delta t}\right) \\
&\leq \left[1+((1-\mu)^2-1)\,\Delta t\right]\,W^2_2\left(\rho^{(1)}_t, \rho^{(2)}_t\right) + \Delta t\,2\,\mu^2\,|m^{(1)}_0 - m^{(2)}_0|\,\expo^{-2\,\mu\,t} \\
&= W^2_2\left(\rho^{(1)}_t, \rho^{(2)}_t\right) + \Delta t\,\left(-(2\,\mu - \mu^2)\,W^2_2\left(\rho^{(1)}_t, \rho^{(2)}_t\right) + 2\,\mu^2\,|m^{(1)}_0 - m^{(2)}_0|\,\expo^{-2\,\mu\,t}\right),
\end{align*}
which leads us to the following differential inequality:
\begin{equation}\label{eq:ODE_W2}
\frac{\dd}{\dd t} W^2_2\left(\rho^{(1)}_t, \rho^{(2)}_t\right) \leq -(2\,\mu - \mu^2)\,W^2_2\left(\rho^{(1)}_t, \rho^{(2)}_t\right) + 2\,\mu^2\,|m^{(1)}_0 - m^{(2)}_0|\,\expo^{-2\,\mu\,t}.
\end{equation}
The desired bound \eqref{eq:pairwise_conv_W2_equal_mu} follows readily from \eqref{eq:ODE_W2}.
\end{proof}

We emphasize here that the quantitative exponential convergence \eqref{eq:conv_W2_equal_mu} of the solution of the mean-field PDE \eqref{eq:main_PDE} in the $W_2$ metric follows essentially from the weak contractivity property (again in $W_2$) of the collision gain operator $Q_+$ \eqref{eq:W2_contra_equal_mu} together with the joint convexity of the squared $W_2$ distance \eqref{eq:E1}. On the other hand, due to the lack of a joint convexity property of the $W_1$ distance, it appears quite challenging to deduce a quantitative estimate on the rate of convergence of $W_1\left(\rho_t, \rho_\infty\right)$ towards zero from Proposition \ref{prop:contra_W1_equal_mu}.

So far our analysis of the large time behavior of the mean-field PDE \eqref{eq:PDE}, especially the proof of Theorem \ref{thm:2}, is based on a operator-theoretic approach where suitable (weak) contractivity properties of the collision gain operator $Q_+$ under certain Wasserstein metrics can be established. However, the aforementioned framework appears restrictive to the (squared) Wasserstein distance of order $2$ as $W^2_2(\cdot,\cdot)$ enjoys the joint convexity property in its arguments while other $W_p$ (for $p \in [1,\infty)$ but $p \neq 2$) does not. The following result extends the $W_2$ convergence reported in Theorem \ref{thm:2} to a convergence guarantee under $W_p$ for any $p \in [1,\infty)$, by virtue of a coupling technique applied directly to the mean-field SDE dynamics \eqref{eq:SDE_linear}.

\begin{theorem}\label{thm:X}
Assume that $\lambda = \mu \in (0,1)$. Let $(\rho^{(1)}_t)_{t \geq 0}$ and $(\rho^{(2)}_t)_{t \geq 0}$ be the solutions to \eqref{eq:PDE} corresponding to initial datum $\rho^{(1)}_0$ and $\rho^{(2)}_0$ with mean values $m^{(1)}_0$ and $m^{(2)}_0$, respectively. Then for all $t\geq 0$ and any $p \geq 1$ we have
\begin{equation}\label{eq:pairwise_conv_Wp_equal_mu}
W^p_p\left(\rho^{(1)}_t, \rho^{(2)}_t\right) \leq W^p_p\left(\rho^{(1)}_0, \rho^{(2)}_0\right)\,\expo^{-(1-(1-\mu)^p)t} + \frac{2^{p-1}\,\left|m^{(1)}_0 - m^{(2)}_0\right|}{(1-(1-\mu)^p)-2\,\mu}\,\left(\expo^{-2\mu t} - \expo^{-(1-(1-\mu)^p)t}\right).
\end{equation}
In particular, for all $t\geq 0$ we have
\begin{equation}\label{eq:conv_Wp_equal_mu}
W^p_p\left(\rho_t, \rho_\infty\right) \leq W^p_p\left(\rho_0, \rho_\infty\right)\,\expo^{-(1-(1-\mu)^p)t} + \frac{2^{p-1}\,|m_0|}{(1-(1-\mu)^p)-2\,\mu}\,\left(\expo^{-2\mu t} - \expo^{-(1-(1-\mu)^p)t}\right).
\end{equation}
\end{theorem}

\begin{proof}
The proof is similar to the proof of Theorem \ref{thm:UPoC} carried out at the level of the agent-based model. Let $Z^{(1)}_t$ and $Z^{(2)}_t$ be the strong solutions to the mean-field SDE \eqref{eq:SDE_linear} employing exactly the same Poisson point measure $\cP(\dd t, \dd u)$, and starting from initial datum $Z^{(1)}_0 \sim \rho^{(1)}_0$ and $Z^{(2)}_0 \sim \rho^{(2)}_0$, respectively. Moreover, we can couple $Z^{(1)}_0$ and $Z^{(2)}_0$ in a optimal way (with respect to $W_p$) to ensure that $\mathbb{E}[|Z^{(1)}_0 - Z^{(2)}_0|^p] = W^p_p(\rho^{(1)}_0,\rho^{(2)}_0)$. Let $h(t) \coloneqq \mathbb{E}[|Z^{(1)}_t - Z^{(2)}_t|^p]$ which serves as a trivial upper bound on $W^p_p(\rho^{(1)}_t,\rho^{(2)}_t)$, a similar estimate as in \eqref{eq:chain_of_estimates} provided in the proof of Theorem \ref{thm:UPoC} leads us to
\begin{equation*}
\begin{aligned}
\frac{\dd}{\dd t} h(t) + h(t) &\leq (1-\mu)^p\,h(t) + 2^{p-1}\,\left|m^{(1)}_t - m^{(2)}_t\right| \\
&= (1-\mu)^p\,h(t) + 2^{p-1}\,\left|m^{(1)}_0 - m^{(2)}_0\right|\,\expo^{-2\mu t}.
\end{aligned}
\end{equation*}
Consequently we deduce that \[\frac{\dd}{\dd t} h(t) \leq -\left(1 - (1-\mu)^p\right)\,h(t) + 2^{p-1}\,\left|m^{(1)}_0 - m^{(2)}_0\right|\,\expo^{-2\mu t}\] and the desired bound \eqref{eq:pairwise_conv_Wp_equal_mu} follows readily from Grönwall's inequality.
\end{proof}

\begin{remark}
In the special case where $p=2$, the estimate \eqref{eq:pairwise_conv_Wp_equal_mu} in Theorem \ref{thm:X} implies that
\begin{equation*}
W^2_2\left(\rho^{(1)}_t, \rho^{(2)}_t\right) \leq \left(W^2_2\left(\rho^{(1)}_0, \rho^{(2)}_0\right) + \frac{2\,\left|m^{(1)}_0 - m^{(2)}_0\right|}{\mu^2}\right)\,\expo^{-(2\mu - \mu^2)\,t}.
\end{equation*}
which is comparable to the previous estimate \eqref{eq:pairwise_conv_W2_equal_mu} reported in Theorem \ref{thm:2}.
\end{remark}

We now show that the weak contractivity of the collision gain operator $Q_+$ can also be established (when $\lambda = \mu$) under a Fourier-based metric which also enjoys the desired the joint convexity property \cite{carrillo_contractive_2007}. For the reader's convenience, we provide a quick review of the so-called Fourier-based distance of order $s \geq 1$ (sometimes also referred to as the Toscani distance of order $s$) \cite{carrillo_contractive_2007}, defined by
\begin{equation}
\label{eq:Toscani_distance}
d_s(f,g) \coloneqq \sup\limits_{\xi \in \mathbb{R}\setminus \{0\}} \frac{|\hat{f}(\xi)-\hat{g}(\xi)|}{|\xi|^s},
\end{equation}
where $f \in \mathcal{P}(\mathbb R)$ and $g \in \mathcal{P}(\mathbb R)$ are probability laws on $\mathbb R$, and
\[\hat{f}(\xi)\coloneqq \int_{\mathbb R} \expo^{-i\,x\,\xi}\,f(\dd x)\] denotes the Fourier transform of $f$. These Fourier-based distances \eqref{eq:Toscani_distance} are introduced in a series of works \cite{carrillo_contractive_2007,gabetta_metrics_1995,goudon_fourier_2002} for the study of the problem of convergence to equilibrium for the spatially homogenous Boltzmann equation stemmed from statistical physics. We emphasize that the Fourier-based distances have also been applied to problems arising from other sub-branches of traditional statistical physics, such as econophysics and sociophysics \cite{cao_bennati_2025,during_boltzmann_2008,matthes_steady_2008,naldi_mathematical_2010}.

We now prove a weak contractivity result of $Q_+$ in $d_1$ when $\lambda = \mu$, which serves as a analogue of Proposition \ref{prop:contra_W2_equal_mu}.

\begin{proposition}[Weak contractivity of $Q_+$ in $d_1$ when $\lambda = \mu$]\label{prop:contra_d1_equal_mu}
Under the settings of Proposition \ref{prop:contra_W1_equal_mu},
\begin{equation}\label{eq:d1_contra_equal_mu}
d_1\left(Q_+[\rho^{(1)}], Q_+[\rho^{(2)}]\right) \leq \mu\,|m^{(1)} - m^{(2)}| + (1-\mu)\,d_1(\rho^{(1)},\rho^{(2)}).
\end{equation}
Consequently, let $(\rho^{(1)}_t)_{t \geq 0}$ and $(\rho^{(2)}_t)_{t \geq 0}$ be the solutions to \eqref{eq:main_PDE} corresponding to initial datum $\rho^{(1)}_0$ and $\rho^{(2)}_0$ with mean values $m^{(1)}_0$ and $m^{(2)}_0$, respectively. Then for $t\geq 0$ it holds that
\begin{equation}\label{eq:d1_contra_equal_mu_PDE}
d_1\left(Q_+[\rho^{(1)}_t], Q_+[\rho^{(2)}_t]\right) \leq \mu\,|m^{(1)}_0 - m^{(2)}_0|\,\expo^{-2\,\mu\,t} + (1-\mu)\,d_1(\rho^{(1)}_t,\rho^{(2)}_t).
\end{equation}
\end{proposition}

\begin{proof}
Using the notations introduced in the proof of \eqref{eq:W1_contra_equal_mu}, we have
\begin{equation}\label{eq:chain_of_inequality}
\begin{aligned}
&d_1\left(Q_+[\rho^{(1)}], Q_+[\rho^{(2)}]\right) \\
&= d_1\left(\mathcal{L}\left(\mu\,(2\,\mathcal{B}^{(1)} - 1) + (1-\mu)\,Z^{(1)}\right),\mathcal{L}\left(\mu\,(2\,\mathcal{B}^{(2)} - 1) + (1-\mu)\,Z^{(2)}\right)\right) \\
&\leq d_1\left(\mathcal{L}\left(\mu\,(2\,\mathcal{B}^{(1)} - 1)\right),\mathcal{L}\left(\mu\,(2\,\mathcal{B}^{(2)} - 1)\right)\right) + d_1\left(\mathcal{L}\left((1-\mu)\,Z^{(1)}\right),\mathcal{L}\left((1-\mu)\,Z^{(2)}\right)\right) \\
&=2\,\mu\,\left|\frac{1-m^{(1)}}{2} - \frac{1-m^{(2)}}{2}\right| + (1-\mu)\,d_1\left(\rho^{(1)},\rho^{(2)}\right),
\end{aligned}
\end{equation}
where the inequality in \eqref{eq:chain_of_inequality} is a consequence of the super-additivity property of $d_1$ with respect to convolution \cite{carrillo_contractive_2007}, and the last identity follows from the explicit computation of $d_1\left(\mathcal{L}\left(\mu\,(2\,\mathcal{B}^{(1)} - 1)\right),\mathcal{L}\left(\mu\,(2\,\mathcal{B}^{(2)} - 1)\right)\right)$ and the scaling property of $d_1$ \cite{carrillo_contractive_2007}.
\end{proof}

\begin{remark}
We remark here that due to the classical fact \cite{carrillo_contractive_2007} that $d_s(f,g) < \infty$ if and only if the laws $f$ and $g$ share the same moments up to order $\floor{s}$ if $s \notin \mathbb N$ or the same moments up to order $s-1$ if $s\in \mathbb N$ (where $\floor{s}$ denotes the integer part of $s$), weak contractivity property of the form \eqref{eq:d1_contra_equal_mu} cannot be true if we work with $d_s$ for $s > 1$ instead of $d_1$. Indeed, one can readily verify via a straightforward computation that
\begin{equation*}
\begin{aligned}
d_s\left(\mathcal{L}\left(\mu\,(2\,\mathcal{B}^{(1)} - 1)\right),\mathcal{L}\left(\mu\,(2\,\mathcal{B}^{(2)} - 1)\right)\right) &= 2\,\left|p^{(1)} -  p^{(2)}\right|\,\sup_{\xi \neq 0} \frac{|\sin(\mu\,\xi)|}{|\xi|^s} \\
&= \begin{cases}
\infty, &~~\textrm{if $s > 1$},\\
2\,\mu\,\left|p^{(1)} -  p^{(2)}\right|, &~~\textrm{if $s = 1$},
\end{cases}
\end{aligned}
\end{equation*}
whenever $\mathcal{B}^{(i)} \sim \textrm{Bernoulli}\left(p^{(i)}\right)$.
\end{remark}

Equipped with the content of Proposition \ref{prop:contra_d1_equal_mu}, we can prove convergence of the solution of the mean-field PDE \eqref{eq:PDE} in the $d_1$ framework.

\begin{theorem}\label{thm:3}
Assume that $\lambda = \mu \in (0,1)$. Let $(\rho^{(1)}_t)_{t \geq 0}$ and $(\rho^{(2)}_t)_{t \geq 0}$ be the solutions to \eqref{eq:PDE} corresponding to initial datum $\rho^{(1)}_0$ and $\rho^{(2)}_0$ with mean values $m^{(1)}_0$ and $m^{(2)}_0$, respectively. Then for all $t\geq 0$ we have
\begin{equation}\label{eq:pairwise_conv_d1_equal_mu}
d_1\left(\rho^{(1)}_t, \rho^{(2)}_t\right) \leq \left(d_1\left(\rho^{(1)}_0, \rho^{(2)}_0\right) + \left|m^{(1)}_0 - m^{(2)}_0\right|\right)\,\expo^{-\mu\,t}.
\end{equation}
In particular, for all $t\geq 0$ we have
\begin{equation*}
	%\label{eq:conv_d1_equal_mu}
	d_1\left(\rho_t, \rho_\infty\right) \leq \left(d_1\left(\rho_0, \rho_\infty\right) + |m_0|\right)\,\expo^{-\mu\,t}.
\end{equation*}
\end{theorem}

\begin{proof}
The proof of \eqref{eq:pairwise_conv_d1_equal_mu} resembles the proof of \eqref{eq:pairwise_conv_W2_equal_mu} from Theorem \ref{thm:2} hence we omit the details. We only need to assemble the joint convexity of the Toscani distance $d_1$ \cite{carrillo_contractive_2007} together with the weak contractivity estimate \eqref{eq:d1_contra_equal_mu_PDE}, to obtain the following differential inequality:
\begin{equation*}
	%\label{eq:ODE_d1}
	\frac{\dd}{\dd t} d_1\left(\rho^{(1)}_t, \rho^{(2)}_t\right) \leq -\mu\,d_1\left(\rho^{(1)}_t, \rho^{(2)}_t\right) + \mu\,|m^{(1)}_0 - m^{(2)}_0|\,\expo^{-2\,\mu\,t},
\end{equation*}
from which the announced bound \eqref{eq:pairwise_conv_d1_equal_mu} follows immediately.
\end{proof}

So far our analysis of the large time behavior of the mean-field PDE \eqref{eq:PDE} has been restricted to the special case where $\lambda = \mu$. In fact, when $\lambda \neq \mu$ it seems impossible to establish the corresponding (weak) contractivity property of the collision gain operator $Q_+$ in $W_2$ or $d_1$. Fortunately, the pathwise coupling method utilized in the proof of Theorem \ref{thm:X} admits a straightforward generalization which enables us to establish a analogue of the large time convergence guarantee \eqref{eq:conv_Wp_equal_mu} even in the case $\lambda \neq \mu$.

\begin{theorem}\label{thm:Y}
Assume that $\lambda < \mu$. Let $(\rho^{(1)}_t)_{t \geq 0}$ and $(\rho^{(2)}_t)_{t \geq 0}$ be the solutions to \eqref{eq:PDE} corresponding to initial datum $\rho^{(1)}_0$ and $\rho^{(2)}_0$ with mean values $m^{(1)}_0$ and $m^{(2)}_0$, respectively. Then for all $t\geq 0$ and any $p \geq 1$, there exists some constant $K$ depending only on $m^{(1)}_0$, $m^{(2)}_0$ and $\lambda, \mu$ such that
\begin{equation}\label{eq:pairwise_conv_Wp_nonequal_mu}
\begin{aligned}
W^p_p\left(\rho^{(1)}_t, \rho^{(2)}_t\right) &\leq W^p_p\left(\rho^{(1)}_0, \rho^{(2)}_0\right)\,\expo^{-(1-(1-\lambda)^p)t} \\
&\quad + \frac{2^{p-1}\,K}{(1-(1-\lambda)^p)-2\,\sqrt{\mu\,\lambda}}\,\left(\expo^{-2\sqrt{\mu\,\lambda}\, t} - \expo^{-(1-(1-\lambda)^p)t}\right).
\end{aligned}
\end{equation}
%where we have set $\mu_{\textrm{min}} \coloneqq \min\{\lambda,\mu\}$.
In particular, for all $t\geq 0$ we have
\begin{equation*}
	%\label{eq:conv_Wp_nonequal_mu}
	\begin{aligned}
	W^p_p\left(\rho_t, \rho_\infty\right) &\leq W^p_p\left(\rho_0, \rho_\infty\right)\,\expo^{-(1-(1-\lambda)^p)t} \\
	&\quad + \frac{2^{p-1}\,K}{(1-(1-\lambda)^p)-2\,\sqrt{\mu\,\lambda}}\,\left(\expo^{-2\sqrt{\mu\,\lambda}\,t} - \expo^{-(1-(1-\lambda)^p)t}\right).
\end{aligned}
\end{equation*}
\end{theorem}

\begin{proof}
The argument follows along the same lines exhibited in the proof of Theorem \ref{thm:X} so we omit the details. In essence, if we let $h(t) \coloneqq \mathbb{E}[|Z^{(1)}_t - Z^{(2)}_t|^p]$, since $\lambda < \mu$, a similar computation as in the proof of Theorem \ref{thm:X} leads us to
\begin{equation*}
\begin{aligned}
\frac{\dd}{\dd t} h(t) + h(t) &\leq \left(1-\lambda\right)^p\,h(t) + 2^{p-1}\,\left|m^{(1)}_t - m^{(2)}_t\right| \\
&\leq (1-\lambda)^p\,h(t) + 2^{p-1}\,K\,\expo^{-2\sqrt{\mu\,\lambda}\,t},
\end{aligned}
\end{equation*}
where the second inequality follows from the explicit formula \eqref{eq:mt_formula}. Thus we arrive at the following differential inequality \[\frac{\dd}{\dd t} h(t) \leq -\left(1-(1-\lambda)^p\right)\,h(t) + 2^{p-1}\,K\,\expo^{-2\sqrt{\mu\,\lambda}\,t},\] from which the estimate bound \eqref{eq:pairwise_conv_Wp_nonequal_mu} follows thanks to Grönwall's inequality.
\end{proof}

\section{Stationary distribution of opinions}
\label{sec:sec4}
\setcounter{equation}{0}
%%%%%%%%%%%%%%%%%%%%%%%%%%%%%%%%%%%%%%%%%%%%%%%%%%%%%%%%%%%%%%%%%%%%%%

\subsection{A generalized Bernoulli convolution}

In this section we focus on the investigation of the stationary distribution $\rho_\infty$ of the mean-field Boltzmann-type PDE \eqref{eq:PDE}, and a slight generalization of it. We begin from the elementary observation, thanks to Lemma \ref{lem:Q_+}, that if $Z_\infty$ is a $\rho_\infty$-distributed random variable, then it must satisfy the following relation:
\[Z_\infty \stackrel{\dd}{=} \mu\,\mathcal{B}_\infty - \lambda\,(1-\mathcal{B}_\infty) + Z_\infty\,\left[(1-\mu)\,\mathcal{B}_\infty + (1-\lambda)\,(1-\mathcal{B}_\infty)\right]\]
or equivalently
\begin{equation}\label{eq:characterization_of_equilibrium}
Z_\infty \stackrel{\dd}{=} \mathcal{B}_\infty\,\left[(1-\mu)\,Z_\infty + \mu\right] + (1-\mathcal{B}_\infty)\,\left[(1-\lambda)\,Z_\infty - \lambda\right],
\end{equation}
where $\stackrel{\dd}{=}$ stands for equality in the sense of distribution. Here $\mathcal{B}_\infty$ is a Bernoulli random variable with parameter $p_\infty = (1-m_\infty)/2$, independent of $Z_\infty$, where we recall that
\[
m_\infty \coloneqq \frac{\sqrt{\mu} - \sqrt{\lambda}}{\sqrt{\mu} + \sqrt{\lambda}},
\qquad \text{thus} \qquad
p_\infty = \frac{\sqrt{\lambda}}{\sqrt{\mu} + \sqrt{\lambda}}.
\]
By a simple contraction argument (which we omit), it can be seen that there exists a unique solution $\rho_\infty = \mathcal{L}(Z_\infty)$ of the equation in distribution \eqref{eq:characterization_of_equilibrium}.

In particular, when $\lambda = \mu \in (0,1)$, the characterization \eqref{eq:characterization_of_equilibrium} simplifies to
\begin{equation}\label{eq:characterization_of_equilibrium_equal_mu}
Z_\infty \stackrel{\dd}{=} (1-\mu)\,Z_\infty + \mu\,\mathcal{R},
\end{equation}
in which $\mathcal{R}$ is a symmetric Rademacher random variable independent of $Z_\infty$, that is, $\mathcal{R}=\pm 1$ with equal probabilities. It is well-known in the fractal geometry literature \cite{erdos_family_1939,kershner_symmetric_1935,solomyak_random_1995,varju_recent_2016,varju_absolute_2019} that the random variable satisfying \eqref{eq:characterization_of_equilibrium_equal_mu} can be expressed in terms of the so-called Bernoulli convolution (sometimes also known as the Rademacher series), which, when $\mu > 1/2$, leads to a Cantor-like fractal structure, in the sense that $\rho_\infty$ is a probability measure on $[-1,1]$ whose support coincides with a Cantor-like set of Lebesgue measure zero. In the context of multi-agent opinion dynamics, the aforementioned fractal structure was rediscovered in the recent work \cite{cao_fractal_2024}, to which we refer the interested reader for a detailed discussion.

In more generality, for $\lambda \neq \mu$, one can consider an equation in distribution like \eqref{eq:characterization_of_equilibrium} with a Bernoulli random variable of arbitrary parameter $p \in (0,1)$. More specifically: denote $\GBC(\lambda,\mu,p)$ (for ``Generalized Bernoulli Convolution'') the law of the random variable $Z$ which is the unique (in distribution) random variable on $[-1,1]$ satisfying
\begin{equation}
\label{eq:def_of_GBC}
	Z \stackrel{\dd}{=}
	\mathcal{B} \,\left[(1-\mu) Z + \mu\right] + (1-\mathcal{B})\,\left[(1-\lambda)\,Z - \lambda\right],
\end{equation}
with $\mathcal{B} \sim \textrm{Bernoulli}(p)$ independent of $Z$. We can write the following explicit formula for $Z \sim \GBC(\lambda,\mu,p)$ (see also \cite{neunhauserer_2001}): given a collection $(\cB_k)_{k=1}^\infty$ of independent $\textrm{Bernoulli}(p)$ random variables, denote $S_k = \sum_{j=1}^k \cB_j$, and set
\begin{equation*}
	%\label{eq:GBC_series}
	%Z = \sum_{k=1}^\infty (1-\mu)^{S_k} (1-\lambda)^{k-S_k} \left[\frac{\mu}{1-\mu}\,\cB_k - \frac{\lambda}{1-\lambda}\,(1-\cB_k)\right].
	Z = \sum_{k=1}^\infty (1-\mu)^{S_{k-1}} (1-\lambda)^{k-1-S_{k-1}} \left[\mu \,\cB_k - \lambda\,(1-\cB_k)\right].
\end{equation*}
It is straightforward to verify that $Z$ solves \eqref{eq:def_of_GBC}.

An equivalent characterization of $\GBC(\lambda,\mu,p)$ using terminologies from \emph{iterative function systems} \cite{falconer_fractal_geometry_1990,hutchinson_fractals_1981,hochman_AnnMath_2014} (IFS) is as follows: consider the maps $T_{0,\lambda}, T_{1,\mu}: [-1,1] \to [-1,1]$ given by
\begin{equation*}
	%\label{eq:IFS}
	T_{0,\lambda}(x) \coloneqq (1-\lambda)\,x-\lambda
	\quad \textrm{and} \quad
	T_{1,\mu}(x) \coloneqq (1-\mu)\,x + \mu.
\end{equation*}
Then the law $\rho \coloneqq \GBC(\lambda,\mu,p)$ is the (unique) probability measure on $[-1,1]$ satisfying
\begin{equation}\label{eq:GBC_equivalent_characterization}
	\rho = (1-p)\,T_{0,\lambda}\# \rho + p\,T_{1,\mu}\# \rho,
\end{equation}
where $T \# \rho \in \mathcal{P}([-1,1])$ denotes the push-forward (probability) measure of $\rho$ by the mapping $T$. Whenever no confusion is possible, we will also abbreviate these maps as $T_0$ and $T_1$, for notational simplicity.

In the fractal geometry literature, the measure $\GBC(\lambda,\mu,p)$ has been considered in \cite{hochman_AnnMath_memoir_2015,neunhauserer_2001,ngai_wang_2005,saglietti_etal_2018} (albeit with a different parameter convention), where it is referred to as a ``non-uniform'' or ``non-homogeneous'' self-similar measure. In the following theorem, we state several analytic, probabilistic, and geometric properties of $\GBC(\lambda,\mu,p)$ that are relevant in our setting. Most of these properties are already well known; nevertheless, for the reader's convenience and for the sake of completeness, we provide self-contained proofs in the Appendix.

\begin{theorem}
	\label{thm:GBC}
	Given $(\lambda,\mu,p) \in (0,1)^3$ and $Z \sim \GBC(\lambda,\mu,p)$, we have:
	\begin{enumerate}[label=(\roman*)]
		
		\item \label{thm:GBC:moments}
		The first and second moment of $Z$ are given by
		\[
		\mathbb{E}[Z] = \frac{p\,\mu - (1-p)\,\lambda}{p\,\mu + (1-p)\,\lambda}
		\]
		and
		\[\mathbb{E}[Z^2] = \frac{2\left(p\,\mu\,(1-\mu)-(1-p)\,\lambda\,(1-\lambda)\right)\mathbb{E}[Z] + p\,\mu^2 + (1-p)\,\lambda^2}{1-p\,(1-\mu)^2-(1-p)\,(1-\lambda)^2}.
		\]
		
		\item \label{thm:GBC:uniform}
		If $p = \lambda = 1-\mu$, then $Z \sim \textrm{Uniform}([-1,1])$. In other words,
		\[
			\GBC(\lambda,1-\lambda,\lambda) = \textrm{Uniform}([-1,1]).
		\]
		
		\item \label{thm:GBC:abs_cont_singular}
		$\GBC(\lambda,\mu,p)$ is either absolutely continuous or singular continuous with respect to the (normalized) Lebesgue measure on $[-1,1]$.
		
		\item \label{thm:GBC:atomless}
		$\GBC(\lambda,\mu,p)$ has no atom.
		
		\item \label{thm:GBC:full_support}
		If $\mu + \lambda \leq 1$, then $\GBC(\lambda,\mu,p)$ has a full support on $[-1,1]$.
		
		\item \label{thm:GBC:Hausdorff}
		If $\lambda + \mu > 1$, then $\GBC(\lambda,\mu,p)$ has a fractal support whose Hausdorff dimension $D \in (0,1)$ is the unique solution of the equation
		\[
			(1-\mu)^D + (1-\lambda)^D = 1.
		\]
		
		\item \label{thm:GBC:mutually_singular}
		$\GBC(\lambda,1-\lambda,p)$ and $\GBC(\lambda,1-\lambda,q)$ are mutually singular whenever $p\neq q$. In particular, $\GBC(\lambda,1-\lambda,p)$ is singular continuous with respect to the (normalized) Lebesgue measure whenever $p\neq \lambda$.
		
		\item \label{thm:GBC:Gaussian}
		Let $X \coloneqq \frac{Z-\mathbb{E}[Z]}{\sigma}$, where $\sigma^2 \coloneqq \Var[Z]$. In the limit $\mu \to 0$ and $\lambda \to 0$ with the ratio $\frac{\mu}{\lambda} = \frac{1-p}{p}$ held fixed, the law of $X$ converges (in distribution) to the standard Gaussian $\mathcal{N}(0,1)$.
		
		\item \label{thm:GBC:W1}
		Given $(\lambda',\mu',p') \in (0,1)^3$, it holds that
		\begin{equation*}
			\begin{aligned}
			%\label{eq:W1_stability}
			& W_1\left(\GBC(\lambda,\mu,p), \GBC(\lambda',\mu',p')\right) \\
			& \leq \frac{2}{p\,\mu + (1-p)\,\lambda}\,\left(|\lambda - \lambda'| + |\mu-\mu'| + |p-p'|\right).
			\end{aligned}
		\end{equation*}
	\end{enumerate}
\end{theorem}

In particular, note the dichotomy stated in \ref{thm:GBC:full_support}-\ref{thm:GBC:Hausdorff}: if $\lambda + \mu > 1$, then $\GBC(\lambda,\mu,p)$ has fractal support; otherwise, its support is the whole interval $[-1,1]$. Interestingly enough, even if $\lambda + \mu \leq 1$, the distribution $\GBC(\lambda,\mu,p)$ can still be singular with respect to the Lebesgue measure. For fixed $p\in(0,1)$, the following result describes the region of singular parameters $(\lambda,\mu) \in (0,1)^2$, except for a negligible set. It is a particular case of \cite[Theorem 1.1]{saglietti_etal_2018}, see also \cite[Theorem 1.15]{hochman_AnnMath_memoir_2015}, so we omit the proof. Define the \emph{similarity dimension} of $\GBC(\lambda,\mu,p)$ by
\[
	s(\lambda,\mu,p)
	\coloneqq \frac{p \log (p) + (1-p)\log(1-p)}{p \log (1-\mu) + (1-p)\log(1-\lambda)}.
\]

\begin{theorem}
	\label{thm:singular_region}
	Fix $p\in(0,1)$. Then $\GBC(\lambda,\mu,p)$ is singular for all $(\lambda,\mu) \in (0,1)^2$ such that $s(\lambda,\mu,p)<1$, and it is absolutely continuous for Lebesgue-almost all $(\lambda,\mu) \in (0,1)^2$ such that $s(\lambda,\mu,p) > 1$.
\end{theorem}

\subsection{Properties of the stationary distribution of opinions}

Let us come back to the equilibrium distribution $\rho_\infty$ of the mean-field PDE \eqref{eq:PDE}, to which end we apply Theorem \ref{thm:GBC} with
\[
	p
	= p_\infty
	= \frac{1-m_\infty}{2}
	= \frac{\sqrt{\lambda}}{\sqrt{\mu}+\sqrt{\lambda}}.
\]

\subsubsection{Case $\lambda + \mu > 1$: opinion fragmentation}

Since $\lim_t \rho_t = \rho_\infty$, guaranteed by Theorem \ref{thm:Y}, point \ref{thm:GBC:Hausdorff} of Theorem \ref{thm:GBC} establishes mathematically the anticipated opinion fragmentation phenomenon of the mean-field model: if the agents' tendency to move toward the extreme political standpoints $\pm 1$ is sufficiently strong---in other words, if $\lambda + \mu > 1$---then the distribution of opinions becomes more and more fragmented over time. In the stationary regime, agents’ opinions concentrate on the fractal set $\supp(\rho_\infty)$ with Hausdorff dimension $D \in (0,1)$ which is the unique solution of the equation
\[
		(1-\mu)^D + (1-\lambda)^D = 1.
\]
	
In the special case when $\lambda = \mu > 1/2$, it is known that $\rho_\infty$ is actually the uniform distribution on its fractal support. For general $\lambda < \mu$ (still assuming $\lambda + \mu > 1$), it is no longer uniform. To better illustrate the behavior of $\rho_\infty$, we performed numerical simulations, whose results we present in Figure \ref{fig:numerics_fractal}. We approximated the evolution of the mean-field PDE \eqref{eq:PDE} for three choices of $(\lambda, \mu)$ with $\lambda + \mu > 1$. We used the uniform distribution in $[-1, 1]$ as the initial datum, and the simulation ran up to time $t = 10$, which seems enough to reach equilibrium. We display an approximation of the cumulative distribution function (CDF) of $\rho_\infty$, by means of a Monte Carlo procedure and by numerically solving the associated PDE. The fractal structure of $\rho_\infty$ can be readily observed: when $\lambda = \mu$ (first row of Figure \ref{fig:numerics_fractal}) we recover the staircase CDF of the Bernoulli convolution, whereas when $\lambda < \mu$ (second and third row) the staircase becomes skewed toward $+1$, which is to be expected since $\EE[Z_\infty] = m_\infty > 0$.
	
\def\figscale{0.625}
\def\figwidth{0.48\textwidth}
\begin{figure}
	% \centering
	\begin{subfigure}{\figwidth}
		\centering
		\includegraphics[scale=\figscale]{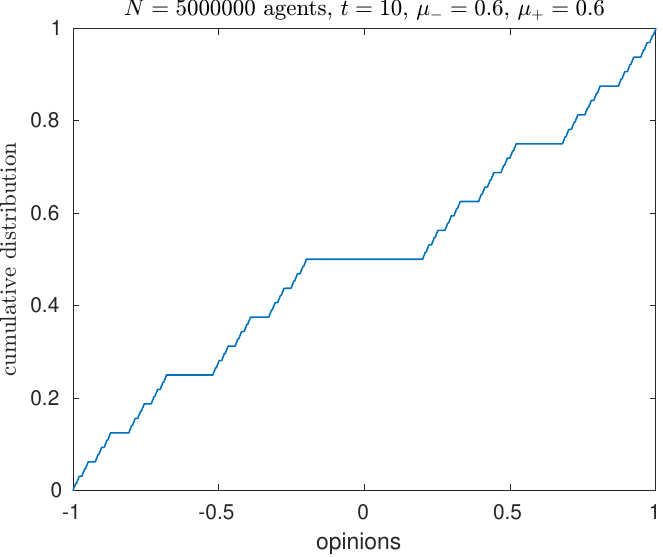}
	\end{subfigure}
	\hspace{0.1in}
	\begin{subfigure}{\figwidth}
		\centering
		\includegraphics[scale=\figscale]{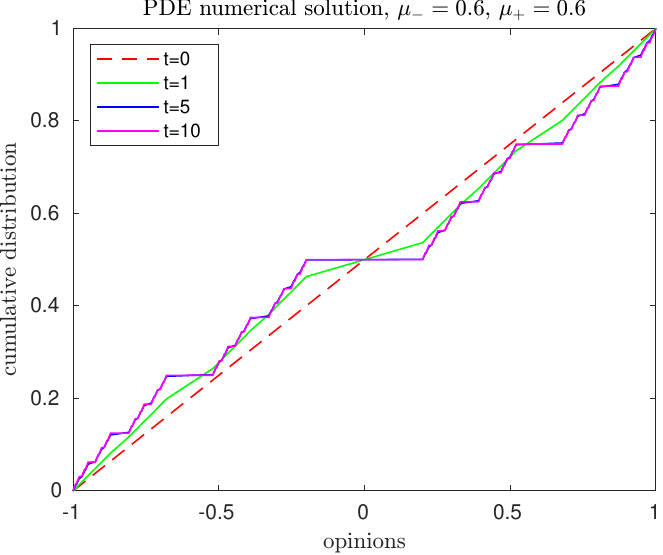}
	\end{subfigure}
	\\
	\begin{subfigure}{\figwidth}
		\centering
		\includegraphics[scale=\figscale]{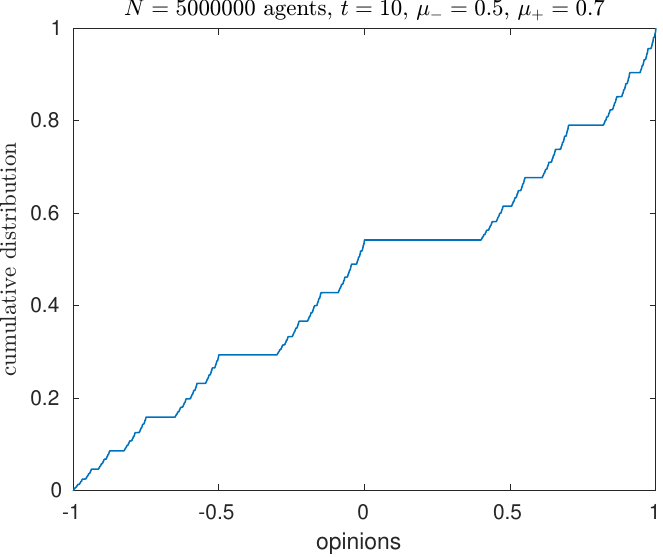}
	\end{subfigure}
	\hspace{0.1in}
	\begin{subfigure}{0.45\textwidth}
		\centering
		\includegraphics[scale=\figscale]{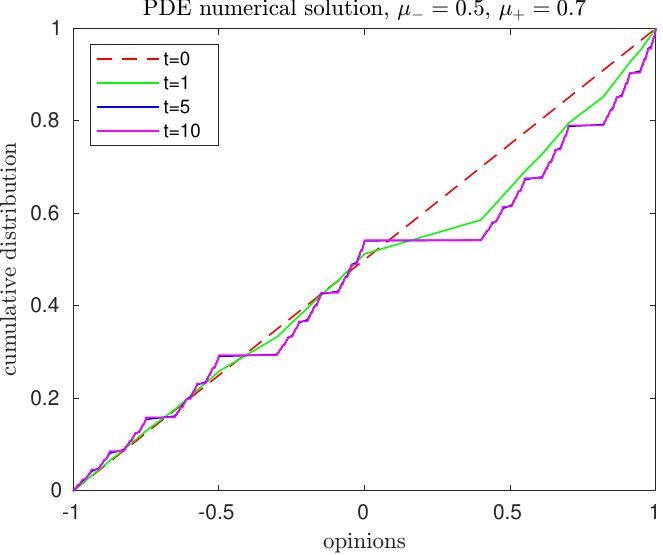}
	\end{subfigure}
	\\
	\begin{subfigure}{\figwidth}
		\centering
		\includegraphics[scale=\figscale]{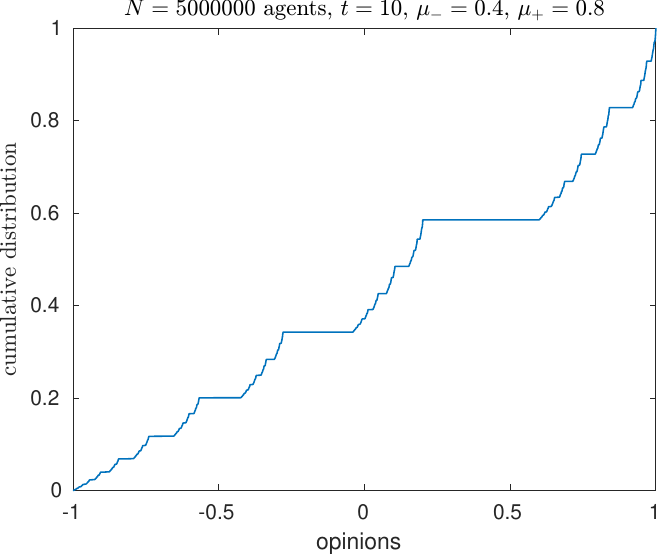}
	\end{subfigure}
	\hspace{0.1in}
	\begin{subfigure}{\figwidth}
		\centering
		\includegraphics[scale=\figscale]{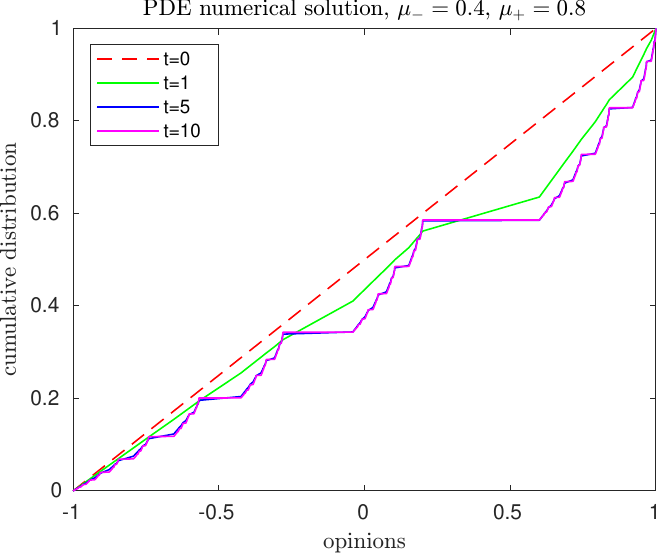}
	\end{subfigure}
	
	\caption{{\bf Left}: Monte Carlo simulation of the agent-based opinion model with $N = 5\cdot 10^6$ agents and time $t=10$, for three pairs of values of $(\lambda, \mu)$ such that $\lambda + \mu > 1$. In each plot the (empirical) cumulative distribution function (CDF) is displayed, which is a good approximation of the CDF of $\rho_\infty$. {\bf Right}: Evolution of the CDF, computed by solving numerically the mean-field PDE \eqref{eq:PDE}, for the same three pairs $(\lambda, \mu)$. The plots exhibit a fractal structure of $\supp(\rho_\infty)$, as predicted by Theorem \ref{thm:GBC}-\ref{thm:GBC:Hausdorff}.}
	\label{fig:numerics_fractal}
\end{figure}

\subsubsection{Case $\lambda + \mu \leq 1$: full support}

On the other hand, when $\lambda + \mu \leq 1$, the scenario is more subtle. Point \ref{thm:GBC:full_support} of Theorem \ref{thm:GBC} ensures that $\rho_\infty$ is supported on $[-1,1]$. The central question is whether $\rho_\infty$ is absolutely continuous with respect to the Lebesgue measure on $[-1,1]$. As a partial answer, we state the following result, which is a direct application of Theorem \ref{thm:singular_region}:
	
\begin{proposition}
	\label{prop:singular}
	Let $(\lambda,\mu) \in (0,1)^2$ be such that $\lambda + \mu \leq 1$ and
	\begin{equation}
		\label{eq:ab}
		a^a\, b^b
		>
		(1-a^2)^{b}\, (1-b^2)^{a}\, (a+b)^{a+b}
		\qquad
		\text{for~ $a = \sqrt{\lambda}$,~ $b = \sqrt{\mu}$}.
	\end{equation}
	Then $\rho_\infty$ has full support on $[-1,1]$ and it is singular continuous with respect to the Lebesgue measure.
\end{proposition}

\begin{proof}
	Thanks to Theorem \ref{thm:singular_region}, we know that $\rho_\infty$ is singular when $s(a^2,b^2,p_\infty)<1$ for $p_\infty = \frac{a}{a+b}$, that is,
	\[
	\frac{\frac{a}{a+b} \log (\frac{a}{a+b}) + \frac{b}{a+b}\log(\frac{b}{a+b})}{\frac{a}{a+b} \log (1-b^2) + \frac{b}{a+b}\log(1-a^2)}
	< 1
	\]
 A straightforward computation shows that this is equivalent to the desired inequality.
\end{proof}

\begin{figure}
	\centering
	\includegraphics[scale=.75]{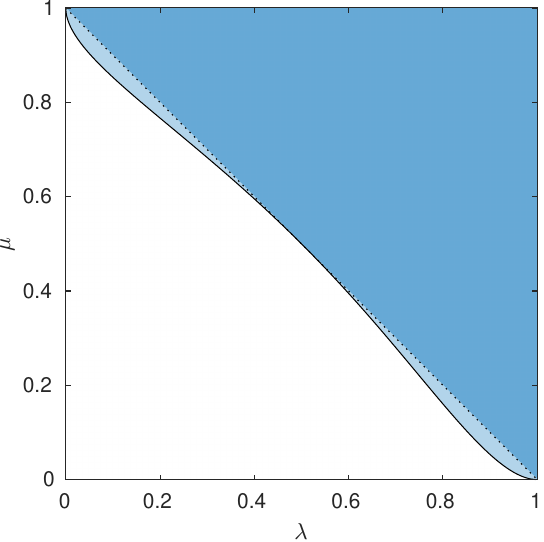}
	\caption{Region of the space $(\lambda,\mu)\in(0,1)^2$ in which $\rho_\infty$ is singular, as described by Proposition \ref{prop:singular}. Above the dotted line $\lambda+\mu=1$, the support of $\rho_\infty$ is fractal. Below it and above the solid line---specified by \eqref{eq:ab}---the support of $\rho_\infty$ is still the full interval $[-1,1]$ but $\rho_\infty$ is still singular.}
	\label{fig:critical_region}
\end{figure}

Thus, there exists a small region on $(0,1)^2$---specified by $\lambda+\mu\leq 1$ and \eqref{eq:ab}---where the support of $\rho_\infty$ is the whole interval $[-1,1]$,  but $\rho_\infty$ is still singular. We display this in Figure \ref{fig:critical_region}. In the remaining region---that is, where \eqref{eq:ab} is not satisfied---it is natural to conjecture that $\rho_\infty$ is absolutely continuous for Lebesgue-almost all such $(\lambda,\mu)$. However, since $p=p_\infty$ is a function of $(\lambda,\mu)$, this cannot be deduced by simply applying Theorem \ref{thm:singular_region}, and we were not able to prove it.

\def\figscale{0.625}
\def\figwidth{0.48\textwidth}
\begin{figure}
	% \centering
	\begin{subfigure}{\figwidth}
		\centering
		\includegraphics[scale=\figscale]{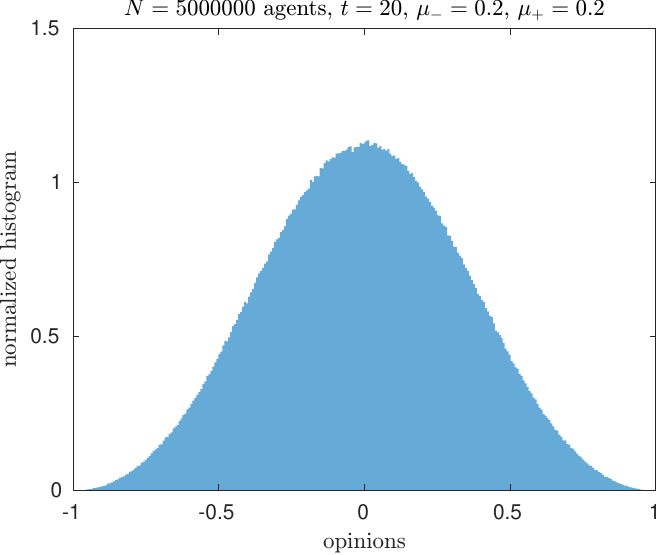}
	\end{subfigure}
	\hspace{0.1in}
	\begin{subfigure}{\figwidth}
		\centering
		\includegraphics[scale=\figscale]{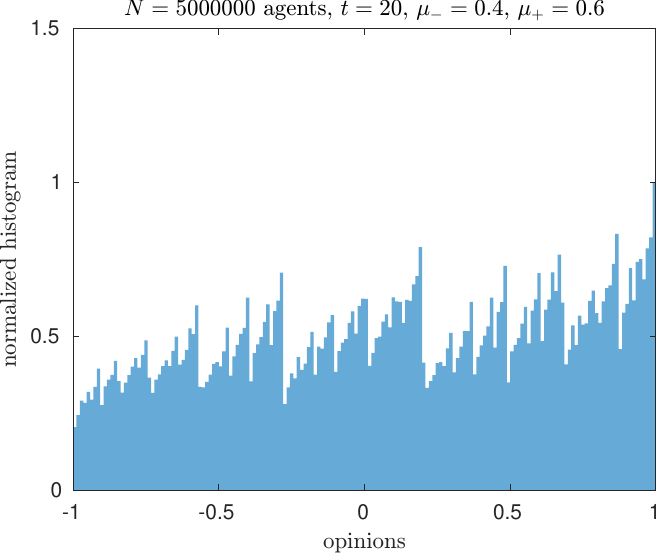}
	\end{subfigure}
	\\
	\begin{subfigure}{\figwidth}
		\centering
		\includegraphics[scale=\figscale]{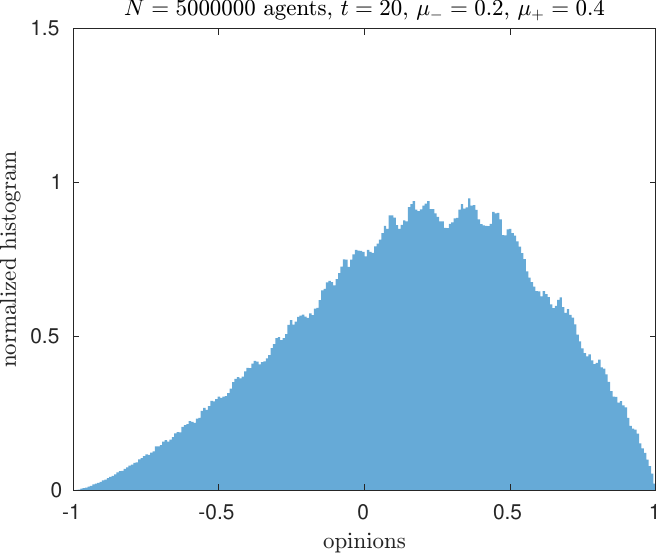}
	\end{subfigure}
	\hspace{0.1in}
	\begin{subfigure}{0.45\textwidth}
		\centering
		\includegraphics[scale=\figscale]{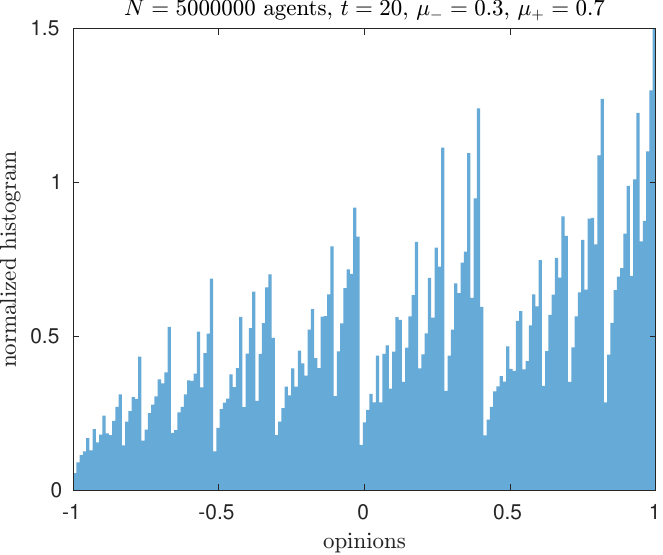}
	\end{subfigure}
	\\
	\begin{subfigure}{\figwidth}
		\centering
		\includegraphics[scale=\figscale]{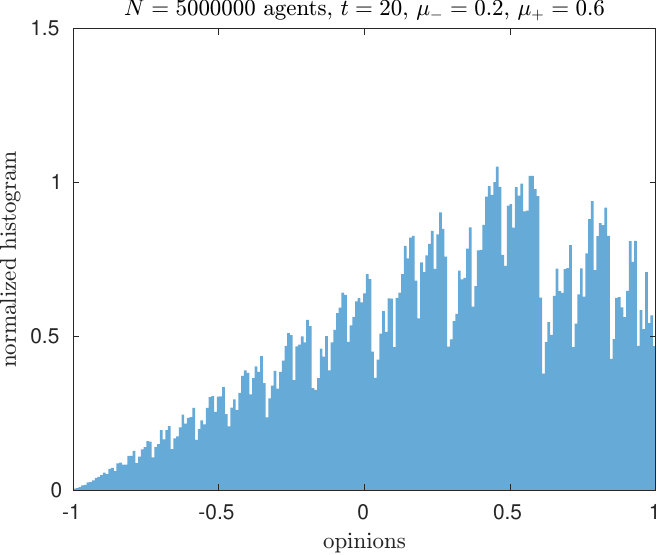}
	\end{subfigure}
	\hspace{0.1in}
	\begin{subfigure}{\figwidth}
		\centering
		\includegraphics[scale=\figscale]{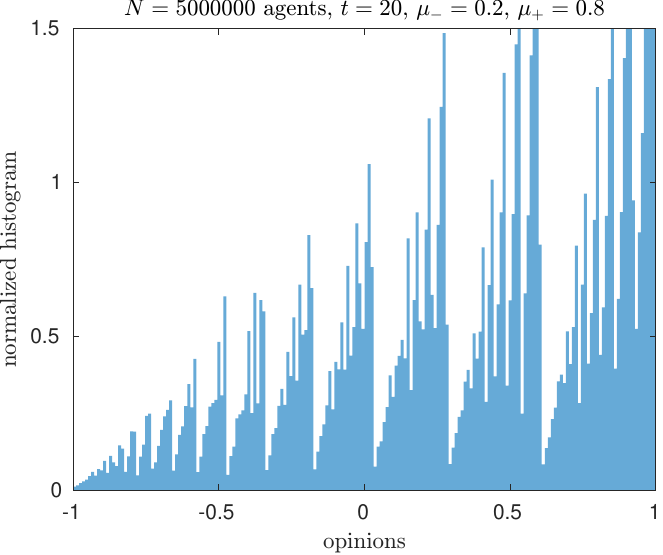}
	\end{subfigure}
	
	\caption{Monte Carlo simulation of the agent-based opinion model with $N = 5\cdot 10^6$ agents and time $t=20$, for six pairs of values of $(\lambda, \mu)$. In each plot the normalized histogram is displayed, as an approximation of the density of $\rho_\infty$, if it exists. {\bf Left:} three pairs of values $(\lambda, \mu)$ such that $\lambda + \mu < 1$. {\bf Right:} three pairs of values $(\lambda, \mu)$ such that $\lambda + \mu = 1$.
	}
	\label{fig:numerics_histogram}
\end{figure}

To better illustrate the behavior of $\rho_\infty$ in terms of $(\lambda,\mu)$, we present some numerical results in Figure \ref{fig:numerics_histogram}. For six choices of $(\lambda, \mu)$ satisfying $\lambda + \mu \leq 1$, we simulated the system of $N$ agents up to time $t=20$, which seems enough to reach equilibrium. We plot the normalized histogram of the agents as an approximation of the density of $\rho_\infty$, if it exists. When $\lambda + \mu < 1$ (left plots), the plot appears to admit a density, especially for small $\lambda$ and $\mu$, even though it becomes rougher for larger values. When $\lambda + \mu = 1$ (right plots), the agents spread across the whole interval $[-1,1]$, but the histogram becomes progressively more irregular as $\mu$ increases.

\section{Conclusion}
\label{sec:sec5}
\setcounter{equation}{0}

In this manuscript, we introduce and study a new opinion dynamics model (inspired from a recent work \cite{cao_fractal_2024}) under the large-population limit $N \to \infty$, where a mean-field description in terms of a Boltzmann-type kinetic equation emerges. We establish quantitative convergence results for the associated Boltzmann-type PDE, showing that its solution approaches a unique equilibrium distribution whose structure is strongly influenced by the model parameters. Notably, in certain parameter regimes, the stationary opinion distribution exhibits an intriguing fractal-like self-similar structure which generalizes the celebrated Bernoulli convolution, which offers a theoretical explanation for the so-called opinion fragmentation phenomenon. Our work highlights a surprising connection between interacting multi-agent systems, kinetic equations, fractal geometry and (generalized) Bernoulli convolutions, underscoring the model's interdisciplinary potential.
The presence of such complexity even in this simplified model highlights the subtlety of opinion formation mechanisms and opens up exciting avenues for further theoretical and numerical exploration. We hope that the tools and perspectives developed in this work will serve as a foundation for future efforts aimed at unraveling the intricate (fractal) geometry of collective behavior in multi-agent systems.

On the other hand, perhaps the most challenging open problem stemming from the present manuscript lies in the systematic study of the generalized Bernoulli convolution $\GBC\big(\lambda,\mu,p(\lambda,\mu)\big)$ associated with the constrained weight
\begin{equation}\label{eq:constrained-weight}
p(\lambda,\mu) \coloneqq \frac{\sqrt{\lambda}}{\sqrt{\mu}+\sqrt{\lambda}}.
\end{equation}
To the best of our knowledge, the fractal geometry community has thus far considered the family $\GBC(\lambda,\mu,p)$ in which the three parameters $(\lambda,\mu,p)\in(0,1)^3$ vary independently of one another. By contrast, the equilibrium measure arising from our opinion dynamics is supported on the two-dimensional surface cut out by the relation \eqref{eq:constrained-weight}, along which the weight $p$ is entirely determined by the pair $(\lambda,\mu)$. Our model thus motivates the investigation of a more intricate object than those studied so far, whose analysis is likely to require advanced tools from analytical/algebraic number theory, fractal geometry, probability theory, and functional analysis.

\bigskip

\noindent {\bf Acknowledgement~} Fei Cao gratefully acknowledges support from an AMS-Simons Travel Grant, administered by the American Mathematical Society with funding from the Simons Foundation. Roberto Cortez acknowledges partial support from Fondecyt Grant 1242001.

\appendix

\section{Appendix}
\setcounter{equation}{0}

\subsection{Proof of Theorem \ref{thm:GBC}}

\begin{proof}
	The proof of \ref{thm:GBC:moments} follows directly from the characterization \eqref{eq:def_of_GBC}. Indeed, taking expectation $\EE[\cdot]$ on both sides of the equation \eqref{eq:def_of_GBC} leads us to
	\begin{equation*}
		\begin{aligned}
			\EE[Z] &= p\left[(1-\mu)\,\EE[Z] + \mu\right] + (1-p)\left[(1-\lambda)\,\EE[Z] - \lambda\right] \\
			&= \left(1-p\,\mu - (1-p)\,\lambda\right)\EE[Z] + p\,\mu - (1-p)\,\lambda,
		\end{aligned}
	\end{equation*}
	from which the claimed expression for $\EE[Z]$ follows. A similar calculation also yields the advertised formula for $\EE[Z^2]$.
	
	We resort to Fourier transform for the proof of \ref{thm:GBC:uniform}. Let $\hat{\rho}(\xi) \coloneqq \EE\left[\expo^{-i\,\xi\,Z}\right]$ denote the Fourier transform of the law $\GBC(\lambda,\mu,p)$, where $\xi \in \R$. Then the identity \eqref{eq:def_of_GBC} yields that
	\begin{equation}\label{eq:Fourier_character}
		\begin{aligned}
			\hat{\rho}(\xi) &= p\,\EE\left[\expo^{-i\,\xi\,\left((1-\mu)\,Z+\mu\right)}\right] + (1-p)\,\EE\left[\expo^{-i\,\xi\,\left((1-\lambda)\,Z-\lambda\right)}\right] \\
			&= p\,\expo^{-i\,\xi\,\mu}\,\hat{\rho}\left((1-\mu)\,\xi\right) + (1-p)\,\expo^{i\,\xi\,\lambda}\,\hat{\rho}\left((1-\lambda)\,\xi\right)
		\end{aligned}
	\end{equation}
	In the special case where $p = \lambda = 1-\mu$, the Fourier characterization \eqref{eq:Fourier_character} of the law $\GBC(\lambda,1-\lambda,\lambda)$ simplifies to
	\begin{equation}\label{eq:Fourier_character_uniform}
		\hat{\rho}(\xi) = (1-\mu)\,\expo^{-i\,\xi\,\mu}\,\hat{\rho}\left((1-\mu)\,\xi\right) + \mu\,\expo^{i\,\xi\,(1-\mu)}\,\hat{\rho}\left(\mu\,\xi\right).
	\end{equation}
	It is straightforward to verify that $\hat{\rho}(\xi) = \frac{\sin(\xi)}{\xi}$, which represents the Fourier transform of the uniform distribution on $[-1,1]$, indeed satisfies the identity \eqref{eq:Fourier_character_uniform}. Hence the proof of \ref{thm:GBC:uniform} is completed.
	
	We now prove \ref{thm:GBC:abs_cont_singular}, that is, we show that $\rho \coloneqq \GBC(\lambda,\mu,p)$ is either absolutely continuous or singular continuous with respect to the Lebesgue measure on $[-1,1]$. Let $\rho = \rho_{\textrm{ac}} + \rho_s$ be the Lebesgue decomposition of $\rho$ with respect to the Lebesgue measure on $[-1,1]$. Since the Lebesgue decomposition is linear under finite positive combinations and is preserved by affine bijections, we deduce from
	\eqref{eq:GBC_equivalent_characterization} that
	\begin{equation*}
		%\label{eq:both}
		\rho_{\textrm{ac}} = (1-p)\,T_0\# \rho_{\textrm{ac}} + p\,T_1\# \rho_{\textrm{ac}} \quad \textrm{and} \quad \rho_s = (1-p)\,T_0\# \rho_s + p\,T_1\# \rho_s.
	\end{equation*}
	Suppose that $\|\rho_{\textrm{ac}}\| \coloneqq \rho_{\textrm{ac}}([-1,1]) > 0$. Then $\bar{\rho}_{\textrm{ac}} \coloneqq \frac{\rho_{\textrm{ac}}}{\|\rho_{\textrm{ac}}\|}$ is a probability measure satisfying the same equation \eqref{eq:GBC_equivalent_characterization} as $\rho$. By uniqueness of the solution to \eqref{eq:GBC_equivalent_characterization}, we deduce that $\bar{\rho}_{\textrm{ac}} = \rho$. As $\bar{\rho}_{\textrm{ac}}$ is absolutely continuous with respect to the Lebesgue measure on $[-1,1]$, so does $\rho$, hence $\rho_s = 0$. Symmetrically, the assumption that $\|\rho_s\| \coloneqq \rho_s([-1,1]) > 0$ forces $\rho_{\textrm{ac}} = 0$, whence the proof of statement \ref{thm:GBC:abs_cont_singular} is completed.
	
	We now tackle \ref{thm:GBC:atomless}: To show that the law $\rho \coloneqq \GBC(\lambda,\mu,p)$ is atomless, we assume that $\rho$ has an atom. Then $\alpha \coloneqq \sup_{x \in [-1,1]} \rho(\{x\}) > 0$ and the set $E \coloneqq \left\{x \in [-1,1] \mid \rho(\{x\}) = \alpha\right\}$ is a non-empty finite set. We pick an arbitrary $x_0 \in E$ and employ the characterization \eqref{eq:GBC_equivalent_characterization} to obtain
	\begin{equation}\label{eq:stationarity_at_x0}
		\alpha = \rho(\{x_0\}) = (1-p)\,\rho(\{T^{-1}_0(x_0)\}) + p\,\rho(\{T^{-1}_1(x_0)\}).
	\end{equation}
	Since $\rho(\{T^{-1}_0(x_0)\}) \leq \alpha$ and $\rho(\{T^{-1}_1(x_0)\}) \leq \alpha$, the equality \eqref{eq:stationarity_at_x0} implies that $\rho(\{T^{-1}_0(x_0)\}) = \rho(\{T^{-1}_1(x_0)\}) = \alpha$. Thus from any maximal atom, both pre-images are again maximal atoms and belong to $[-1,1]$. Upon iteration, $T^{-n}_1(x_0) = 1 + (1-\mu)^{-n}\,(x_0 - 1)$ must lie in $[-1,1]$ for all $n\in \mathbb N$, which cannot hold true unless $x_0 = 1$. Similarly, $T^{-n}_0(x_0) = -1 + (1-\lambda)^{-n}\,(x_0 + 1)$ must lie in $[-1,1]$ for all $n\in \mathbb N$, which forces $x_0 = -1$. Because $x_0 = 1$ and $x_0 = -1$ cannot hold simultaneously, we reach a contradiction, whence $\alpha = 0$ and the law $\rho$ has no atom.

	We now tackle \ref{thm:GBC:full_support}-\ref{thm:GBC:Hausdorff}. Denote $\cS \coloneqq \supp\left(\GBC(\lambda,\mu,p)\right)$. From \eqref{eq:def_of_GBC}, we see that $\cS$ satisfies
	\begin{equation}
		\label{eq:supp_S}
		\cS = \left[(1-\mu)\,\cS + \mu\right] \cup \left[(1-\lambda)\,\cS - \lambda\right].
	\end{equation}
	Properties \ref{thm:GBC:full_support}-\ref{thm:GBC:Hausdorff} follow from the self-similarity condition \eqref{eq:supp_S}, by applying well-known results in the literature of fractal geometry; see for instance \cite{hutchinson_fractals_1981} for details.
	
	To prove \ref{thm:GBC:mutually_singular}, we rely on the strong law of large numbers and an argument inspired by the monograph \cite{bishop_fractals_2017}. The key component consists of investigating how the measure $\GBC(\lambda,1-\lambda,p)$ assigns mass to cylinder sets. Let $I = [-1,1]$ and recall the two contractions $T_0, T_1 \colon I \to I$ given by $T_0(x) = \mu\, x - \lambda$ and $T_1(x) = \lambda\, x + \mu$. Let $\Omega^n \coloneqq \{0,1\}^n$ be the set of all binary strings of length $n \in \mathbb{N}_+$. For a given binary sequence $\sigma = (\sigma_1,\ldots,\sigma_n) \in \Omega^n$, the cylinder set $I_\sigma$ is defined as the image of the base interval $I$ under the composition of the maps corresponding to $\sigma$:
	\begin{equation*}
		I_\sigma \coloneqq T_{\sigma_1} \circ T_{\sigma_2} \circ \cdots \circ T_{\sigma_n}(I).
	\end{equation*}
	By virtue of an induction argument, we see that for each $n \geq 1$, the collection of cylinders $\{I_\sigma \colon \sigma \in \Omega^n\}$ consists of $2^n$ intervals whose interiors are disjoint and whose union is exactly $[-1,1]$. Since the measure $\GBC(\lambda,1-\lambda,p)$ satisfies
	\begin{equation*}
		\GBC(\lambda,1-\lambda,p) = (1-p)\GBC(\lambda,1-\lambda,p)\circ T^{-1}_0 + p\GBC(\lambda,1-\lambda,p)\circ T^{-1}_1
	\end{equation*}
	and the cylinders at level $n$ have disjoint interiors, we deduce that
	\begin{equation*}
		\GBC(\lambda,1-\lambda,p)\left\{I_\sigma\right\} = (1-p)^{n_0}\,p^{n_1} \quad \text{and} \quad \GBC(\lambda,1-\lambda,q)\left\{I_\sigma\right\} = (1-q)^{n_0}\,q^{n_1},
	\end{equation*}
	where $n_0$ and $n_1$ represent the number of $0$s and $1$s appearing in the binary string $\sigma \in \Omega^n$.
	
	Recall that every point $x \in [-1,1]$ can be associated with an infinite binary sequence $\sigma \in \Omega^{\mathbb{N}}$ representing its path through the cylinders. Let $N_1(x,n)$ be the number of $1$'s in the first $n$ symbols of the representation of $x$ and define the set $E_\alpha$ as the set of points where the asymptotic frequency of positive pulls is $\alpha \in (0,1)$:
	\begin{equation*}
		E_\alpha \coloneqq \left\{x \in [-1,1] \colon \lim_{n \to \infty} \frac{N_1(x,n)}{n} = \alpha \right\}.
	\end{equation*}
	Since the collection $\{\sigma_i\}$ consists of independent Bernoulli random variables with parameter $p$ with respect to $\GBC(\lambda,1-\lambda,p)$ and parameter $q$ with respect to $\GBC(\lambda,1-\lambda,q)$, the strong law of large numbers implies that
	\begin{equation*}
		\GBC(\lambda,1-\lambda,p)\left\{E_p\right\} = 1 \quad \text{and} \quad \GBC(\lambda,1-\lambda,q)\left\{E_q\right\} = 1.
	\end{equation*}
	Because $p \neq q$ by assumption, the sets $E_p$ and $E_q$ are disjoint, whence the measures $\GBC(\lambda,1-\lambda,p)$ and $\GBC(\lambda,1-\lambda,q)$ are mutually singular.
	
	To prove \ref{thm:GBC:Gaussian}, we notice that the relation $\frac{\mu}{\lambda} = \frac{1-p}{p}$ guarantees that $X$ is a mean zero random variable with unit variance. Thus if $\hat{f}(\xi) = \coloneqq \EE\left[\expo^{-i\,\xi\,X}\right]$ denotes the Fourier transform of the law $f$ of $X$, we have $\hat{f}(0) = 1$, $\hat{f}'(0) = -i\,\EE[X] = 0$, and $\hat{f}''(0) = -\EE[X^2] = -1$. Moreover, a straightforward computation in the same spirit as the derivation of \eqref{eq:Fourier_character} shows that the Fourier characterization of $f$ is given by
	\begin{equation}\label{eq:Fourier_character_X}
		\hat{f}(\xi) = p\,\expo^{-i\,\xi\,\frac{\mu}{\sigma}}\,\hat{f}\left((1-\mu)\,\xi\right) + (1-p)\,\expo^{i\,\xi\,\frac{\lambda}{\sigma}}\,\hat{f}\left((1-\lambda)\,\xi\right).
	\end{equation}
	Next, we observe that in the asymptotic regime where $\mu \ll 1$ and $\lambda \ll 1$ with $\frac{\mu}{\lambda} = \frac{1-p}{p}$ held fixed, it holds that $\frac{\mu^2}{\sigma^2} = \cO(\mu)$ and $\frac{\lambda^2}{\sigma^2} = \cO(\lambda)$. Consequently, by Taylor expanding the right-hand side of \eqref{eq:Fourier_character_X} (and ignoring higher order terms), we end up with
	\begin{equation*}
		%\label{eq:Gaussian_Fourier}
		\begin{aligned}
			\lim_{\substack{\mu,\lambda \to 0 \\ \frac{\mu}{\lambda} = \frac{1-p}{p}}} \frac{\hat{f}'(\xi)}{\hat{f}(\xi)} &= \lim_{\substack{\mu,\lambda \to 0 \\ \frac{\mu}{\lambda} = \frac{1-p}{p}}} \frac{-\frac 12\,\xi\,\frac{1}{\sigma^2}\,\left[p\,\mu^2 + (1-p)\,\lambda^2\right]}{p\,\mu + (1-p)\,\lambda} \\
			&= \lim_{\substack{\mu,\lambda \to 0 \\ \frac{\mu}{\lambda} = \frac{1-p}{p}}} \frac{-\frac 12\,\xi\,\left[1-p\,(1-\mu)^2 - (1-p)\,(1-\lambda)^2\right]}{p\,\mu + (1-p)\,\lambda} \\
			&= -\frac{\xi}{2}\,\lim_{\mu \to 0} \frac{1-p\,(1-\mu)^2-(1-p)\,\left(1-\frac{\mu\,p}{1-p}\right)^2}{2\,\mu\,p} \\
			&= -\frac{\xi}{2}\,\lim_{\mu \to 0} \frac{-2\,p\,(\mu-1)-2\,(1-p)\,\left(\frac{\mu\,p}{1-p}-1\right)\,\frac{p}{1-p}}{2\,p} \\
			&= -\frac{\xi}{2}\,\lim_{\mu \to 0} \left[1-\mu + 1-\frac{\mu\,p}{1-p}\right] = -\xi.
		\end{aligned}
	\end{equation*}
	As a result, in the limit $\mu \to 0$ and $\lambda \to 0$ while keeping the ratio $\frac{\mu}{\lambda} = \frac{1-p}{p}$ frozen, the Fourier transform of the law $f$ of $X$ converges to the solution of the first-order separable ODE $\frac{\hat{f}'(\xi)}{\hat{f}(\xi)} = -\xi$, which leads us to $\hat{f}(\xi) = \expo^{-\frac{\xi^2}{2}}$ (i.e., the Fourier transform of the standard normal distribution). This completes the proof of the statement \ref{thm:GBC:Gaussian}.
	
	Finally, we prove \ref{thm:GBC:W1}. We let ${\bf \theta} = (\lambda,\mu,p)$ and introduce the functional $F_{\bf \theta} \colon \mathcal{P}([-1,1]) \to \mathcal{P}([-1,1])$ by setting
	\begin{equation*}
		%\label{def:operator}
		F_{\bf \theta}[\nu] \coloneqq (1-p)\,T_{0,\lambda} \# \nu + p\,T_{1,\mu}\# \nu
	\end{equation*}
	for $\nu \in \mathcal{P}([-1,1])$. Then the law $\nu_{\bf \theta} \coloneqq \GBC(\lambda,\mu,p)$ is the unique fixed point of $F_{\bf \theta}$. We first prove the following contraction estimate
	\begin{equation}\label{eq:contraction_W1}
		W_1\left(F_{\bf \theta}[\nu], F_{\bf \theta}[\nu'] \right) \leq \left(1-p\,\mu - (1-p)\,\lambda\right)\,W_1(\nu,\nu')
	\end{equation}
	which holds for all $\nu, \nu' \in \mathcal{P}([-1,1])$. Indeed, assume that $Z \sim \nu$, $Z' \sim \nu'$, and $\mathcal{B}_p$ is a Bernoulli random variable (independent of $Z$ and $Z'$) with parameter $p$, then
	\begin{equation*}
		\mathcal{B}_p\,[(1-\mu)\,Z + \mu] + (1-\mathcal{B}_p)\,[(1-\lambda)\,Z - \lambda] \sim F_{\bf \theta}[\nu]
	\end{equation*}
	and \begin{equation*}
		\mathcal{B}_p\,[(1-\mu)\,Z' + \mu] + (1-\mathcal{B}_p)\,[(1-\lambda)\,Z' - \lambda] \sim F_{\bf \theta}[\nu'].
	\end{equation*}
	Therefore, we deduce that
	\begin{equation*}
		%\label{eq:chain_of_estimates_2}
		\begin{aligned}
			W_1\left(F_{\bf \theta}[\nu], F_{\bf \theta}[\nu'] \right) &\leq \EE\left|\cB_p\,(1-\mu)\,(Z-Z') + (1-\cB_p)\,(1-\lambda)\,(Z-Z')\right|\\
			&\leq \EE[\cB_p\,(1-\mu) + (1-\cB_p)\,(1-\lambda)]\cdot \EE|Z-Z'| \\
			&= (p\,(1-\mu) + (1-p)\,(1-\lambda))\,\EE|Z-Z'| \\
			&= \left(1-p\,\mu - (1-p)\,\lambda\right)\,\EE|Z-Z'|,
		\end{aligned}
	\end{equation*}
	whence the claimed estimate \eqref{eq:contraction_W1} follows by choosing the optimal coupling between $Z$ and $Z'$ with respect to $W_1$. Now we denote ${\bf \theta}' = (\lambda',\mu',p')$ and claim that
	\begin{equation}\label{eq:parameter_perturbation}
		W_1\left(F_{\bf \theta}[\nu], F_{{\bf \theta}'}[\nu] \right) \leq 2\,(|\lambda-\lambda'| + |\mu-\mu'| + |p-p'|)
	\end{equation}
	for all $\nu \in \mathcal{P}([-1,1])$. Indeed, we introduce the intermediate measure defined by
	\begin{equation*}
		%\label{eq:M_measure}
		\cM_{\lambda',\mu',p}\,\nu \coloneqq p\,T_{1,\mu'}\# \nu + (1-p)\,T_{0,\lambda'}\# \nu.
	\end{equation*}
	Thanks to the triangle inequality, we have
	\[W_1\left(F_{\bf \theta}[\nu], F_{{\bf \theta}'}[\nu] \right) \leq W_1\left(F_{\bf \theta}[\nu], \cM_{\lambda',\mu',p}\,\nu\right) + W_1\left(\cM_{\lambda',\mu',p}\,\nu, F_{{\bf \theta}'}[\nu]\right).\]
	Since for any $z \in [-1,1]$ it holds that $|T_{0,\lambda}(z) - T_{0,\lambda'}(z)| \leq 2\,|\lambda-\lambda'|$ and $|T_{1,\mu}(z) - T_{1,\mu'}(z)| \leq 2\,|\mu-\mu'|$, a probabilistic coupling argument similar to the derivation of \eqref{eq:contraction_W1} gives rise to
	\begin{equation*}
		%\label{eq:IE1}
		W_1\left(F_{\bf \theta}[\nu], \cM_{\lambda',\mu',p}\,\nu\right) \leq 2\,|\lambda-\lambda'| + 2\,|\mu-\mu'|.
	\end{equation*}
	On the other hand, for any $1$-Lipschitz continuous function $\phi \colon [-1,1] \to \mathbb R$, we observe that
	\begin{equation}\label{eq:observation}
		\int_{-1}^1 \phi \, \dd \left(\cM_{\lambda',\mu',p}\,\nu - F_{{\bf \theta}'}[\nu]\right) = (p-p')\,\int_{-1}^1 \phi \, \dd \left(T_{1,\mu'}\# \nu - T_{0,\lambda'}\# \nu\right).
	\end{equation}
	Using the Kantorovich–Rubinstein duality formulation of the $W_1$ distance and taking $\phi$ to a Kantorovich potential so that
	\begin{equation*}
		W_1\left(T_{1,\mu'}\# \nu, T_{0,\lambda'}\# \nu\right) = \int_{-1}^1 \phi \, \dd \left(T_{1,\mu'}\# \nu - T_{0,\lambda'}\# \nu\right),
	\end{equation*}
	we deduce from the identity \eqref{eq:observation} that
	\begin{equation}\label{eq:IE2}
		W_1\left(\cM_{\lambda',\mu',p}\,\nu, F_{{\bf \theta}'}[\nu]\right) \leq |p-p'|\,W_1\left(T_{1,\mu'}\# \nu, T_{0,\lambda'}\# \nu\right) \leq 2\,|p-p'|,
	\end{equation}
	where the last inequality in \eqref{eq:IE2} follows from the fact the interval $[-1,1]$ is compact with diameter $2$. As $\nu_{\bf \theta} = F_{\bf \theta}[\nu_{\bf \theta}]$ and $\nu_{{\bf \theta}'} = F_{{\bf \theta}'}[\nu_{{\bf \theta}'}]$, we apply the triangle inequality together with the contraction estimate \eqref{eq:contraction_W1} to obtain
	\begin{align*}
		W_1(\nu_{\bf \theta}, \nu_{{\bf \theta}'}) &\leq W_1\left(F_{\bf \theta}[\nu_{\bf \theta}], F_{\bf \theta}[\nu_{{\bf \theta}'}]\right) + W_1\left(F_{\bf \theta}[\nu_{{\bf \theta}'}], F_{{\bf \theta}'}[\nu_{{\bf \theta}'}]\right) \\
		&\leq \left(1-p\,\mu - (1-p)\,\lambda\right)\,W_1(\nu_{\bf \theta}, \nu_{{\bf \theta}'}) + \sup_{\nu \in \cP([-1,1])} W_1\left(F_{\bf \theta}[\nu], F_{{\bf \theta}'}[\nu]\right).
	\end{align*}
	Therefore, we conclude the proof by using the estimate \eqref{eq:parameter_perturbation} as follows:
	\begin{equation*}
		\begin{aligned}
			W_1(\nu_{\bf \theta}, \nu_{{\bf \theta}'}) &\leq \frac{1}{p\,\mu + (1-p)\,\lambda}\,\sup_{\nu \in \cP([-1,1])} W_1\left(F_{\bf \theta}[\nu], F_{{\bf \theta}'}[\nu]\right) \\
			&\leq \frac{2}{p\,\mu + (1-p)\,\lambda}\,(|\lambda-\lambda'| + |\mu-\mu'| + |p-p'|).
		\end{aligned}
	\end{equation*}
\end{proof}


\begin{thebibliography}{99}

\bibitem{barnidge_2018}
Matthew Barnidge.
\newblock Social affect and political disagreement on social media.
\newblock {\em Social Media + Society}, \textbf{4}(3):2056305118797721, 2018.

\bibitem{bennaim_2005}
Eli Ben-Naim.
\newblock Opinion dynamics: rise and fall of political parties.
\newblock {\em Europhysics Letters}, \textbf{69}(5):671--677, 2005.

\bibitem{bishop_fractals_2017}
Christopher J. Bishop, and Yuval Peres.
\newblock Fractals in probability and analysis.
\newblock \emph{Cambridge University Press}, 2017.

\bibitem{bobylev_generalization_1992}
Alexander V. Bobylev, and Giuseppe Toscani.
\newblock On the generalization of the Boltzmann H-theorem for a spatially homogeneous Maxwell gas.
\newblock {\em Journal of Mathematical Physics}, \textbf{33}(7):2578--2586, 1992.

\bibitem{cao_bennati_2025}
Fei Cao, and Nadia Loy.
\newblock The Bennati-Dragulescu-Yakovenko model in the continuous setting: PDE derivation and long-time behavior.
\newblock {\em arXiv preprint arXiv:2512.06101}, 2025.

\bibitem{cao_derivation_2021}
Fei Cao, and Sebastien Motsch.
\newblock Derivation of wealth distributions from biased exchange of money.
\newblock {\em Kinetic \& Related Models}, \textbf{16}(5):764--794, 2023.

\bibitem{cao_entropy_2021}
Fei Cao, Pierre-Emannuel Jabin, and Sebastien Motsch.
\newblock Entropy dissipation and propagation of chaos for the uniform reshuffling model.
\newblock {\em Mathematical Models and Methods in Applied Sciences}, \textbf{33}(4):829--875, 2023.

\bibitem{cao_explicit_2021}
Fei Cao.
\newblock Explicit decay rate for the Gini index in the repeated averaging model.
\newblock {\em Mathematical Methods in the Applied Sciences}, \textbf{46}(4):3583--3596, 2023.

\bibitem{cao_fractal_2024}
Fei Cao, and Roberto Cortez.
\newblock Fractal opinions among interacting agents.
\newblock {\em SIAM Journal on Applied Dynamical Systems}, \textbf{24}(2):1529--1552, 2025.

\bibitem{cao_interacting_2022}
Fei Cao, and Pierre-Emannuel Jabin.
\newblock From interacting agents to Boltzmann-Gibbs distribution of money.
\newblock {\em Nonlinearity}, \textbf{37}(12):125020, 2024.

\bibitem{cao_k_2021}
Fei Cao.
\newblock $K$-averaging agent-based model: propagation of chaos and convergence to equilibrium.
\newblock {\em Journal of Statistical Physics}, \textbf{184}(2):18, 2021.

\bibitem{cao_iterative_2024}
Fei Cao, and Stephanie Reed.
\newblock The iterative persuasion-polarization opinion dynamics and its mean-field analysis.
\newblock {\em SIAM Journal on Applied Mathematics}, \textbf{85}(4):1596--1620, 2025.

\bibitem{cao_uncovering_2022}
Fei Cao, and Sebastien Motsch.
\newblock Uncovering a two-phase dynamics from a dollar exchange model with bank and debt.
\newblock {\em SIAM Journal on Applied Mathematics}, \textbf{83}(5):1872--1891, 2023.

\bibitem{cao_uniform_2024}
Fei Cao, and Roberto Cortez.
\newblock Uniform propagation of chaos for a dollar exchange econophysics model.
\newblock {\em European Journal of Applied Mathematics}, \textbf{36}(1):27--39, 2025.

\bibitem{carrillo_contractive_2007}
Jos{\'e} A. Carrillo, and Giuseppe Toscani.
\newblock Contractive probability metrics and asymptotic behavior of dissipative kinetic equations.
\newblock \emph{Riv. Mat. Univ. Parma}, \textbf{6}(7):75--198, 2007.

\bibitem{castellano_statistical_2009}
Claudio Castellano, Santo Fortunato, and Vittorio Loreto.
\newblock Statistical physics of social dynamics.
\newblock {\em Reviews of modern physics}, \textbf{81}(2):591, 2009.

\bibitem{cortez_quantitative_2016}
Roberto Cortez, and Joaquin Fontbona.
\newblock Quantitative propagation of chaos for generalized Kac particle systems.
\newblock {\em The Annals of Applied Probability}, \textbf{26}(2):892--916, 2016.

\bibitem{cortez_uniform_2016}
Roberto Cortez.
\newblock Uniform propagation of chaos for Kac's 1D particle system.
\newblock {\em Journal of Statistical Physics}, \textbf{165}:1102--1113, 2016.

\bibitem{cortez_fontbona_2018}
Roberto Cortez, and Joaquin Fontbona.
\newblock Quantitative uniform propagation of chaos for Maxwell molecules.
\newblock {\em Communications in Mathematical Physics}, \textbf{357}(3):913--941, 2018.


\bibitem{goyanes_etal_2021}
Manuel Goyanes, Porismita Borah, and Homero Gil de Zúñiga.
\newblock Social media filtering and democracy: Effects of social media news use and uncivil political discussions on social media unfriending.
\newblock {\em Computers in Human Behavior}, \textbf{120}:106759, 2021.

\bibitem{deffuant_mixing_2000}
Guillaume Deffuant, David Neau, Frederic Amblard, and G{\'e}rard Weisbuch.
\newblock Mixing beliefs among interacting agents.
\newblock {\em Advances in Complex Systems}, \textbf{3}(01n04):87--98, 2000.

\bibitem{degond_macroscopic_2004}
Pierre Degond.
\newblock Macroscopic limits of the Boltzmann equation: a review.
\newblock {\em Modeling and computational methods for kinetic equations}, pp. 3--57, 2004.

\bibitem{during_boltzmann_2008}
Bertram D{\"u}ring, Daniel Matthes, and Giuseppe Toscani.
\newblock A Boltzmann-type approach to the formation of wealth distribution curves.
\newblock \emph{Available at SSRN 1281404}, 2008.

\bibitem{erdos_family_1939}
Paul Erd{\"o}s.
\newblock On a family of symmetric Bernoulli convolutions.
\newblock {\em American Journal of Mathematics}, \textbf{61}(4):974--976, 1939.


\bibitem{falconer_fractal_geometry_1990}
\newblock Kenneth Falconer.
\newblock Fractal geometry: mathematical foundations and applications.
\newblock John Wiley \& Sons, Ltd., Chichester, 1990.


\bibitem{gabetta_metrics_1995}
G. Gabetta, Giuseppe Toscani, and Bernt Wennberg.
\newblock Metrics for probability distributions and the trend to equilibrium for solutions of the Boltzmann equation.
\newblock \emph{Journal of Statistical Physics}, \textbf{81}:901--934, 1995.

\bibitem{galam_gefen_shapir_1982}
Serge Galam, Yuval Gefen, and Yonathan Shapir.
\newblock Sociophysics: A new approach of sociological collective behavior.
\newblock \emph{The Journal of Mathematical Sociology}, \textbf{9}(1):1--13, 1982.

\bibitem{goudon_fourier_2002}
Thierry Goudon, St{\'e}phane Junca, and Giuseppe Toscani.
\newblock Fourier-based distances and Berry-Esseen like inequalities for smooth densities.
\newblock \emph{Monatshefte f{\"u}r Mathematik}, \textbf{135}:115--136, 2002.

\bibitem{graham_meleard_1997}
Carl Graham, and Sylvie M{\'e}l{\'e}ard.
\newblock Stochastic particle approximations for generalized Boltzmann models and convergence estimates.
\newblock \emph{The Annals of Probability}, \textbf{25}(1):115--132, 1997.

\bibitem{hegselmann_opinion_2002}
Rainer Hegselmann, and Ulrich Krause.
\newblock Opinion dynamics and bounded confidence models, analysis, and simulation.
\newblock \emph{Journal of artificial societies and social simulation}, \textbf{5}(3), 2002.

\bibitem{heinrich_conformity_2025}
Elisa Heinrich Mora, Kaleda K. Denton, Michael E. Palmer, and Marcus W. Feldman.
\newblock Conformity to continuous and discrete ordered traits.
\newblock {\em Proceedings of the National Academy of Sciences}, \textbf{122}(03):e2417078122, 2025.

\bibitem{hochman_AnnMath_2014}
Michael Hochman.
\newblock On self-similar sets with overlaps and inverse theorems for entropy.
\newblock {\em Annals of Mathematics}, \textbf{180}:773–822, 2014.

\bibitem{hochman_AnnMath_memoir_2015}
Michael Hochman.
\newblock On self-similar sets with overlaps and inverse theorems for entropy in $\mathbb{R}^d$.
\newblock {\em arXiv preprint arxiv:1503.09043}, 2015.

\bibitem{holley_ergodic_1975}
Richard A. Holley, and Thomas M. Liggett.
\newblock Ergodic theorems for weakly interacting infinite systems and the voter model.
\newblock {\em The Annals of Probability}, \textbf{3}(4):643--663, 1975.

\bibitem{hutchinson_fractals_1981}
John E.\ Hutchinson.
\newblock Fractals and self similarity.
\newblock {\em Indiana University Mathematics Journal}, \textbf{30}(5):713--747, 1981.

\bibitem{jabin_clustering_2014}
Pierre-Emmanuel Jabin, and Sebastien Motsch.
\newblock Clustering and asymptotic behavior in opinion formation.
\newblock {\em Journal of Differential Equations}, \textbf{257}(11):4165--4187, 2014.

\bibitem{jessen_wintner_1935}
Børge Jessen, and Aurel Wintner.
\newblock Distribution functions and the Riemann zeta function.
\newblock {\em Transactions of the American Mathematical Society}, \textbf{38}:48--88, 1935.

\bibitem{kershner_symmetric_1935}
Richard Kershner, and Aurel Wintner.
\newblock On symmetric Bernoulli convolutions.
\newblock {\em American Journal of Mathematics}, \textbf{57}(3):541--548, 1935.


\bibitem{liggett_interacting_1985}
\newblock Thomas Milton Liggett, and Thomas M. Liggett.
\newblock Interacting particle systems.
\newblock New York: Springer, 1985.

\bibitem{matthes_steady_2008}
Daniel Matthes, and Giuseppe Toscani.
\newblock On steady distributions of kinetic models of conservative economies.
\newblock {\em Journal of Statistical Physics}, \textbf{130}(6):1087--1117, 2008.

\bibitem{naldi_mathematical_2010}
Giovanni Naldi, Lorenzo Pareschi, and Giuseppe Toscani.
\newblock Mathematical modeling of collective behavior in socio-economic and life sciences.
\newblock \emph{Springer Science \& Business Media}, 2010.

\bibitem{neunhauserer_2001}
Jörg Neunhäuserer.
\newblock Properties of Some Overlapping Self-Similar and Some Self-Affine Measures.
\newblock \emph{Acta Mathematica Academiae Scientiarum Hungaricae}, \textbf{92}(1):143-161, 2001.


\bibitem{ngai_wang_2005}
Sze-Man Ngai, and Yang Wang.
\newblock Self-similar measures associated to IFS with non-uniform contraction ratios.
\newblock \emph{Asian Journal of Mathematics}, \textbf{9}(2), 227-244, 2005.


\bibitem{pandey_etal_2023}
Sriniwas Pandey, Yiding Cao, Yingjun Dong, Minjun Kim, Neil G. MacLaren, Shelley D. Dionne, Francis J. Yammarino, and Hiroki Sayama
\newblock Generation and influence of eccentric ideas on social networks.
\newblock {\em Scientific Reports}, \textbf{13}(1):20433, 2023.


\bibitem{saglietti_etal_2018}
Santiago Saglietti, Pablo Shmerkin, and Boris Solomyak.
\newblock Absolute continuity of non-homogeneous self-similar measures.
\newblock {\em Advances in Mathematics}, \textbf{335}:60--110, 2018.


\bibitem{sen_sociophysics_2014}
Parongama Sen, and Bikas K. Chakrabarti.
\newblock Sociophysics: an introduction.
\newblock OUP Oxford, 2014.

\bibitem{solomyak_random_1995}
Boris Solomyak.
\newblock On the random series $\sum \pm \lambda^n$ (an Erd{\"o}s problem).
\newblock {\em Annals of Mathematics}, \textbf{142}(3):611--625, 1995.

\bibitem{sznajd_opinion_2000}
Katarzyna Sznajd-Weron, and Jozef Sznajd.
\newblock Opinion evolution in closed community.
\newblock {\em International Journal of Modern Physics C}, \textbf{11}(06):1157--1165, 2000.

\bibitem{sznitman_topics_1991}
Alain-Sol Sznitman.
\newblock Topics in propagation of chaos.
\newblock In {\em Ecole d'été de probabilités de {Saint}-{Flour}
  {XIX}—1989}, pages 165--251. Springer, 1991.

\bibitem{toscani_2006}
Giuseppe Toscani.
\newblock Kinetic models of opinion formation.
\newblock {\em Communications in Mathematical Sciences}, \textbf{4}(3):481--496, 2006.

\bibitem{varju_recent_2016}
P{\'e}ter P. Varj{\'u}.
\newblock Recent progress on Bernoulli convolutions.
\newblock {\em European Congress of Mathematics}, pp. 847--867, 2016.

\bibitem{varju_absolute_2019}
P{\'e}ter P. Varj{\'u}.
\newblock Absolute continuity of Bernoulli convolutions for algebraic parameters.
\newblock {\em Journal of the American Mathematical Society}, \textbf{32}(2):351--397, 2019.


\bibitem{zhang_shoenberger_2024}
Bingbing Zhang, and Heather Shoenberger.
\newblock Navigating political disagreement on social media: How affective responses and belonging influence unfollowing and unfriending.
\newblock {\em Media and Communication,} \textbf{12}(0), 2024.


\end{thebibliography}
\end{document}